\documentclass[11pt]{amsart}
\usepackage{amsmath,amssymb,amsthm,latexsym,cite,cancel}
\usepackage[small]{caption}
\usepackage{graphicx,wasysym,overpic,tikz,color,xfrac}
\usepackage{subfigure,color}
\usepackage{cite}
\usepackage[colorlinks=true,urlcolor=blue,
citecolor=red,linkcolor=blue,linktocpage,pdfpagelabels,
bookmarksnumbered,bookmarksopen]{hyperref}
\usepackage[italian,english]{babel}
\usepackage{units}
\usepackage{enumitem}
\usepackage[left=2.1cm,right=2.1cm,top=2.71cm,bottom=2.71cm]{geometry}
\usepackage[hyperpageref]{backref}
\usepackage{ulem}

\usepackage[colorinlistoftodos]{todonotes}
\newtheorem{theorem}{Theorem}[section]
\newtheorem{definition}[theorem]{Definition}
\newtheorem{proposition}[theorem]{Proposition}
\newtheorem{lemma}[theorem]{Lemma}
\newtheorem{remark}[theorem]{Remark}

\newtheorem{corollary}[theorem]{Corollary}

\newcommand{\abs}[1]{\left\lvert#1\right\rvert}
\newcommand{\norm}[1]{\lVert#1\rVert}
\newcommand{\red}[1]{\textcolor{red}{#1}}

\numberwithin{equation}{section}

\title[The double phase problems]{The double phase problems on lattice graphs}

\author[Z.T. He]{Zhentao He}

\address[Z.T. He]{\newline\indent
	School of Mathematics
	\newline\indent
	East China University of Science and Technology
	\newline\indent
	Shanghai 200237, PR China }
\email{\href{mailto:hezhentao2001@outlook.com}{hezhentao2001@outlook.com}}

\author[C. Ji]{Chao Ji}

\address[C. Ji]{\newline\indent
	School of Mathematics
	\newline\indent
	East China University of Science and Technology
	\newline\indent
	Shanghai 200237, PR China }
\email{\href{mailto:jichao@ecust.edu.cn}{jichao@ecust.edu.cn}}

\date{\today}

\begin{document}

\begin{abstract}
  In this paper, we first develop the theory of Musielak-Orlicz spaces on locally finite graphs, including completeness, reflexivity, separability, and so on. Then, we give some elementary properties of double phase operators on locally finite graphs.
   Finally, as applications of previous theory, we prove some existence results of solutions to double phase problems on lattice graphs.
\end{abstract}

\maketitle

\begin{center}
	\begin{minipage}{11cm}
		\tableofcontents
	\end{minipage}
\end{center}

\section{Introduction and main results}
\label{secin}

In the past decades, a considerable amount of literature has been devoted to double phase problems. These problems are motivated by various physical phenomena. Firstly, the double phase operator has been utilized to describe steady-state solutions of reaction-diffusion problems in various fields such as biophysics, plasma physics, and chemical reaction analysis. The prototype equation for these models can be expressed as:
$$
u_t=\Delta_p^a u(z)+\Delta_q u+g(x, u) \quad \text{ in } \Omega_T,
$$
where $\Omega_T :=  \Omega \times (0, T)$ denotes the space-time cylinder, and $\Omega$ denotes a bounded domain in $\mathbb{R}^n$ with $n \geqslant 2$ and $t\in (0, T)$ with $T>0$. In this context, the function $u$ typically represents a concentration, the term $\Delta_p^a u(z)+\Delta_q u$ accounts for diffusion with coefficients determined by $a(z)|\nabla u|^{p-2}+|\nabla u|^{q-2}$, and $g(x, u)$ describes the reaction term associated with source and loss processes. For further details, please refer to Cherfils \& Il'yasov  \cite{CIi}  and Singer  \cite{S}.  Secondly, such operators offer a valuable framework for explaining the behavior of highly anisotropic materials. These materials exhibit varying hardening properties, which are influenced by the exponent governing the propagation of the gradient variable. These variations occur depending on the point of interest, where the modulating coefficient $a(z)$ influences the geometry of a composite composed of two different materials.  For instance, this structure is used to describe a composite material having on $\{z \in \Omega: \alpha(z)=0\}$ an energy density with $q$-growth, but on $\{z \in \Omega: \alpha(z)>0\}$ the energy density has $p$-growth.

This operator is related to the double phase energy functional, which is defined by
$$
u \rightarrow \int_{\Omega}\Big(a(z)\vert\nabla u\vert^p+\vert\nabla u\vert^q\Big)dz.
$$
It is noteworthy that if the weight function $a(\cdot)$ is not bounded away from zero, i.e.,  $\underset{\Omega}{\text{ess}\inf}\,a\geq 0$, the density function of the above integral functional, denoted as the integrand $\eta(z, t)=a(z) t^p+t^q$, exhibits unbalanced growth, which can be characterized by:
$$
t^q\leq \eta(z, t)\leq c_{0}(1+t^p),\,\,\text {for a.e.}\,\, z\in\Omega, \,\,\text {all}\,\, t\geq 0,\,\,\text {some}\,\, c_{0}>0.
$$
Such integral functionals were first considered by Marcellini \cite{Ma, Ma1} and Zhikov \cite{Zh1, Zh2} in the context of problems in the calculus of variations (including the Lavrentiev gap phenomenon) and of nonlinear elasticity theory.  For problems with unbalanced growth, only local regularity results exist (please see the survey papers \cite{Ma2} due to Marcellini and \cite{MR} due to Mingione and R\v{a}dulescu),  and there is no global regularity theory (i.e., regularity up to the boundary). This limitation restricts the use of many powerful technical tools available for problems with balanced growth and the nonlinearity $f$ satisfies some subcritical growth conditions. For further results about existence, multiplicity, regularity and qualitative properties of the solutions to double phase problems, please see \cite{He,Amb,CIi,liu2018existence,Po,Col1,Col2,LiuP,Pap,Fang,Fil} and the references therein.


In recent years, the nonlinear Schr\"{o}dinger equations on graphs has received considerable attention (see \cite{hua2023existence,Hua1,Hua2,Gri0,Gri1,Chang,Zhang,Shao,Han,Shao1,hua2021existence,Hua,Stefanov,Weinstein} and the references therein).
In particular, in the monograph \cite{Gri0}, Grigor'yan introduced discrete Laplace operator on finite and infinite graphs. In \cite{Gri1}, Grigor'yan, Lin and Yang considered the following nonlinear Schr\"{o}dinger equation
\begin{equation}\label{eqsch}
    -\Delta u + h(x)u = f (x, u)
\end{equation}
on locally finite graph $G=(V,E)$, where $f$ satisfies some growth conditions and the potential function $h$ satisfies $1/h\in\ell^1(V)$ or $h(x) \to\infty$ as $d(x, x_0) \to \infty$ for some $x_0 \in V$. Under the assumptions on $h$, they proved Sobolev compact embeddings and obtained the existence of positive solution to equation \eqref{eqsch} via variational methods. In \cite{hua2023existence}, applying the Nehari method, Hua and Xu studied the existence of ground states of equation \eqref{eqsch} on the lattice graphs, where $f$ satisfies some growth conditions and the potential function $h$ is periodic or bounded, and they extended their results to quasi-transitive graphs.

The following Sobolev inequality in lattice graph $\mathbb{Z}^N$ is well-known
\begin{equation}\label{eqsobolev}
    \|u\|_{r} \leqslant C_{r}\|\nabla u\|_p, \quad \forall u \in D^{1,p}\left(\mathbb{Z}^{N}\right),
\end{equation}
where $N \geqslant 2$, $1 \leqslant p < N$, $r \geqslant p^* := \frac{Np}{N-p}$, see \cite[Theorem 3.6]{hua2015time} for the detailed proof. In \cite{hua2021existence}, by minimizing $\|\nabla u\|_p^p$ on $\left\{ u\in  D^{1,p}(\mathbb{Z}^N) : \norm{u}_r = 1\right\}$, Hua and Li proved the existence of extremal function for \red{the} Sobolev inequality \eqref{eqsobolev}, and then they obtained a positive solution to the equation
\[
    -\Delta_p u = u^{r-1}\quad \text{ in } \mathbb{Z}^N,
\]
where $N \geqslant 3$, $1 \leqslant p < N$, $r > p^*$.

However, as far as we know, there are no any
results on double phase problems on graphs. In this paper,  we will make some significant attempts and establish some fundamental theories and results in this direction. More precisely, we are studying the double phase problems on the lattice graphs for the first time.
To investigate these problems, we first develop the theory of Musielak-Orlicz spaces on locally finite graphs, including completeness,
  reflexivity, separability, and so on. Then, we give some elementary properties of double phase operators on locally finite graphs.
  We think that these have independent interest and can be applied to the study of related problems.

In the continuous case $\mathbb{R}^N$, the following double phase problem has been extensively studied
\begin{equation}\label{eqrn}
\begin{cases}
-\mathrm{div}(\abs{\nabla u}^{p-2}\nabla u+a(x)\abs{\nabla u}^{q-2}\nabla u) = g(u), & \text { in } \mathbb{R}^N, \\
u(x) \to 0 & \text{ as } \abs{x} \to +\infty,
\end{cases}
\end{equation}
where $N \geqslant 3$, $1 < p < N$ and $p < q$. The assumption that $a(x)$ is a positive constant is crucial to obtain the compactness of Palais-Smale sequences, see \cite{Po,He}. When $a(x)$ is not a positive constant, to the best of our knowledge, there is nothing about the existence of solutions of \eqref{eqrn} in the literature. The present paper
is the first contribution to the double phase problem on lattice graphs under different assumptions on $a$. Moreover, our work space is continuously embedded in $\ell^r(\mathbb{Z}^N)$ for all $r \geqslant p^*$, while it is continuously embedded in  $L^r(\mathbb{R}^N)$ for all $r\in [p^*,q^*]$ in continuous setting.

Now we recall some basic setting on graphs from \cite{hua2021existence, Zhang}. Let $G = (V, E)$ be a connected locally finite graph,
where $V$ denotes the vertex set and $E \subset V\times V$ denotes the edge set. A graph is called locally finite if each vertex has finitely many neighbours and is called connected if for any two vertices $x,y \in V$, there is a path (with finite length) connecting $x$ and $y$. The graph distance $d(x,y)$ of two vertices $x, y \in V$ is defined by the minimal number of edges which connect these two vertices. Denote $C(V)$ as the set of all real functions defined on $V$.

Throughout this paper, we regard $E$ as a subset of $V \times V$. Therefore, for any $(x,y)\in E$, $(x,y)$ and $(y,x)$ are different elements in $E$ while they stand for the same edge. Let $\mu \colon V \to (0, +\infty)$ be a measure on $V$, and $w \colon V\times V\to [0, +\infty)$ be an edge weight function satisfying
 \[w(x,y) = w(y,x), \quad \text{for all }  (x,y)\in E,\]  and for any $(x,y)\in V \times V$, \[(x,y) \in E \Leftrightarrow w(x,y)>0.\]
 To simplify the notation, in the following, if $(x,y) \in E$, we also write $xy \in E$ or $y \sim x$, and define $w_{xy}= w(x,y)$.

The $N$-dimensional integer lattice graph,
denoted by $\mathbb{Z}^{N}$, is the graph consisting of the set of
vertices $V=\mathbb{Z}^{N}$ and the set of edges
\[
E=\left\{ (x,y) :x,y\in\mathbb{Z}^{N},\mathop{\sum\limits _{i=1}^{N}\abs{x_{i}-y_{i}}=1}\right\} .
\]
For any $x \in \mathbb{Z}^N$, define $\abs{x} := d(x,0)$. On lattice graphs, we always assume that $\mu \equiv 1$ and $w \equiv 1$.

In this paper, we shall study the following double phase problem on the lattice graph
\begin{equation}
  \label{eq1}
  -\mathrm{div}(\abs{\nabla u}^{p-2}\nabla u+a(x)\abs{\nabla u}^{q-2}\nabla u)=f(x,u)\quad \text{ in } \mathbb{Z}^N,
\end{equation}
where $N \geqslant 2$, $1<p<N$, $p<q<\infty$, $a \in C(\mathbb{Z}^N)$, $a(x) \geqslant 0$ for all $x \in \mathbb{Z}^{N}$,
and the double phase operator on the lattice graph is denoted by
    \begin{equation*}
  \begin{aligned}
    L(u)(x_0):
    = & -\mathrm{div}(\abs{\nabla u}^{p-2}\nabla u+a(x)\abs{\nabla u}^{q-2}\nabla u)(x_0)\\
    = & -\sum_{y\sim{x_0}}\abs{u(y)-u(x_0)}^{p-2}(u(y)-u(x_0))\\
      & -\sum_{y\sim{x_0}}\frac{a(x_0)+a(y)}{2}\abs{u(y)-u(x_0)}^{q-2}(u(y)-u(x_0))
      \end{aligned}
\end{equation*}
for all $x_0 \in \mathbb{Z}^N$ and $u \in C(\mathbb{Z}^{N})$.  It is natural to consider the integral
\begin{equation}
\label{eqintegral}
        I(u) := \frac{1}{2} \sum_{xy \in E} \left(\frac{1}{p}\abs{u(y)-u(x)}^p + \frac{a(x)}{q}\abs{u(y)-u(x)}^q\right)
\end{equation}
to be the functional corresponding to the double phase operator $L$. If $a \in \ell^\infty(\mathbb{Z}^N)$, then by Remark \ref{remarkequi} below, $I$ is well-defined and of class $C^1$ on $D^{1,p}(\mathbb{Z}^N)$. However, if $a$ is unbounded, then, for each $1 \leqslant \alpha \leqslant \infty$, \eqref{eqintegral} could be infinite for some $u \in D^{1,\alpha}(\mathbb{Z}^N)$, and thus $I$ is no longer a well-defined functional on $D^{1,\alpha}(\mathbb{Z}^N)$. Therefore, we introduce the Musielak-Orlicz space $D^{1,\mathcal{H}}(\mathbb{Z}^N)$ in Subsection \ref{ZN} to ensure \eqref{eqintegral} is well defined and of class $C^1$ on $D^{1,\mathcal{H}}(\mathbb{Z}^N)$. Therefore, we can deal with problem $\eqref{eq1}$ under different assumptions on $a$ uniformly.

Now, we present the main results of this paper.
\begin{theorem}
  \label{theorem1}
  If $f(x,t)=f(x)$, $f \in \ell^{r}(\mathbb{Z}^N)$ and $\frac{1}{r}+\frac{1}{p^*} \geqslant 1$,
   then problem \eqref{eq1} has a unique solution.
\end{theorem}
\begin{remark}
  Since there is no similar results in $\mathbb{R}^N$, we compare Theorem \ref{theorem1} to \cite[Theorem 1.1]{liu2018existence} (double phase problems on bounded domain $\Omega \subset \mathbb{R}^N$). In Theorem \ref{theorem1}, the assumption on $r$ is $\frac{1}{r}+\frac{1}{p^*} \leqslant 1$, because we have the embedding $\ell^{p^*}(\mathbb{Z}^N) \hookrightarrow \ell^{\frac{r}{r-1}}(\mathbb{Z}^N)$ for all $\frac{r}{r-1} \geqslant p^*$ by Proposition \ref{proembed}(i) below. While in \cite[Theorem 1.1]{liu2018existence}, the assumption on $r$ is $\frac{1}{r}+\frac{1}{p^*} \leqslant 1$, because $L^{p^*}(\Omega) \hookrightarrow L^{\frac{r}{r-1}}(\Omega)$ for all $1 \leqslant \frac{r}{r-1} \leqslant p^*$.
\end{remark}
Next, let $r > p^*$. Define
\begin{equation*}
\begin{aligned}
J \colon D^{1,\mathcal{H}}(\mathbb{Z}^N) &\to \mathbb{R}\\
 u &\to \frac{1}{r}\int_{\mathbb{Z}^N}{\abs{u}^r}\, d\mu := \frac{1}{r}\sum_{x \in \mathbb{Z}^N}\abs{u(x)}^r.
  \end{aligned}
\end{equation*}
For given $t > 0$, let
\[
    M_{t} := \left\{ u \in D^{1,\mathcal{H}}(\mathbb{Z}^N) : J(u) = t \right\}.
\]
We consider the following constrained minimization problem
\begin{equation}
    \label{eqconstrained}
      S_{t} := \inf_{u \in M_{t}} I(u).
\end{equation}
$u_0$ is called a solution of problem \eqref{eqconstrained} if  $\in M_{t}$ and $I(u_0) = \inf_{u \in M_{t}} I(u)$.

Moreover, due to the theorem of Lagrange multipliers and the integration by parts formula \eqref{eqgreen}, if $u$ is a solution of the problem \eqref{eqconstrained}, then there exists $\lambda > 0$ such that $\langle I'(u),v \rangle = \lambda \langle J'(u),v \rangle$ for all $v \in D^{1,\mathcal{H}}(\mathbb{Z}^N)$, that is, $u$ is a nontrivial weak solution of the eigenvalue problem
\begin{equation}
    \label{eq2}
    -\mathrm{div}(\abs{\nabla u}^{p-2}\nabla u+a(x)\abs{\nabla u}^{q-2}\nabla u)= \lambda \abs{u}^{r-2}u \quad \text{ in } \mathbb{Z}^N.
\end{equation}
For any fixed $y \in \mathbb{Z}^N$, take the test function $v$ to be
    \[
    \delta_y(x):=
     \begin{cases}
       1,  &x=y,\\
       0, &x \neq y,
      \end{cases}
    \]
    then we obtain
    \begin{equation*}
      -\mathrm{div}(\abs{\nabla u}^{p-2}\nabla u+a(x)\abs{\nabla u}^{q-2}\nabla u)(y)= \lambda \abs{u}^{r-2}u(y),
    \end{equation*}
    and thus $u$ is also a pointwise solution of \eqref{eq2}.

\begin{theorem}
\label{theoremcompact}
    Let $ r > p^*$. If $a$ satisfies
    \[
        \lim_{\abs{x} \to \infty}a(x) = + \infty,
    \]
  then for each $t > 0$, problem \eqref{eqconstrained} (or \eqref{eq2}) has a positive solution $u$.
\end{theorem}

We observe that Theorem \ref{theorem1} and \ref{theoremcompact} can be extended to more general graphs. Let $G = (V,E)$ be a uniformly locally finite graph satisfying $\inf\limits_{x \in V}\mu(x) > 0$ and $d$-isoperimetric inequality for some $d \geqslant 2$ (see \cite{hua2015time} for more details of uniformly locally finite graphs and $d$-isoperimetric inequality), then Theorem \ref{theorem1} holds for $f \in \ell^{r}(V)$, where $\frac{1}{r}+\frac{1}{p^*} \geqslant 1$ and $p^*:= \frac{dp}{d-p}$. Because in our proof, the only property of $\mathbb{Z}^N$ we used is $D^{1,\mathcal{H}}(\mathbb{Z}^N) \hookrightarrow \ell^{r}(\mathbb{Z}^N)$ for all $r \geqslant p^*$, see Proposition \ref{proembed}(v) below. If we further assume that $G$ satisfies $\inf\limits_{xy \in E}w_{xy} > 0$, then Theorem \ref{theoremcompact} holds by replacing $\abs{x}$ with $d(x,x_0)$ for some $x_0 \in V$.

\begin{theorem}
  \label{theoremmain}
  Let $ r > p^*$, $r \geqslant q$ and $a \in \ell^\infty(\mathbb{Z}^N)$.
  Let $\{ v_n \} \subset D^{1,\mathcal{H}}(\mathbb{Z}^{N})$ be a minimizing sequence of \eqref{eqconstrained} for any $t > 0$, and $\{y_n\} \subset \mathbb{Z}^{N}$ be an arbitrary sequence ensuring $\{x_n - y_n\} $ bounded in $\mathbb{Z}^N$, where $\{x_n\} \subset \mathbb{Z}^{N}$ such that $\abs{v_n(x_n)} = \norm{v_n}_\infty$. Then $\{u_n(x) := v_n(x+y_n)\}$ contains a convergent subsequence, still denoted by $\{u_n\}$, that converges to $u \in D^{1,\mathcal{H}}(\mathbb{Z}^{N})$.
  Moreover, $u$ satisfies
  \begin{equation}
      \label{eqT1}
      \int_E{\left(\frac{1}{p}\abs{\nabla u}^p+ \frac{\tilde{a}(x)}{q}\abs{\nabla u}^q\right)}\, dw =S_{t}
  \end{equation}
    and solves
\begin{equation}
    \label{eqtildea}
    -\mathrm{div}(\abs{\nabla u}^{p-2}\nabla u+\tilde{a}(x)\abs{\nabla u}^{q-2}\nabla u)= \lambda \abs{u}^{r-2}u\quad \text{ in } \mathbb{Z}^N,
\end{equation}
where $\tilde{a}(x) = \lim\limits_{n \to \infty}{a(x+y_n)}$ and $\lambda$ is a positive constant. Moreover, if $v_n$ are non-negative, then $u$ is positive.
\end{theorem}
The existence of $\tilde{a}$ is provided in Lemma \ref{lemkey}. In general, $I(u)$ is not equal to $S_{t}$ since $\tilde{a}$ varies depending on the choice of $\{y_n\}$ in Theorem \ref{theoremmain}. However, under some certain assumptions on $a$, we can prove that $S_{t}$ is achieved at some $u \in D^{1,\mathcal{H}}(\mathbb{Z}^{N})$.

For $T \in \mathbb{N}^+$, we call that a function $b$ on $\mathbb{Z}^N$ be $T$-periodic, if
\[
    b(x + Te_i) = b(x), \quad \forall x \in \mathbb{Z}^N, 1\leqslant i \leqslant N,
\]
where $e_i$ is the unit vector in the $i$-th coordinate.

We call that a function $b$ on $\mathbb{Z}^N$ be bounded potential, if
    \[
        0 \leqslant b(x) \leqslant b_\infty:=\lim_{\abs{x} \to \infty}b(x) < \infty, \quad
        \forall x \in \mathbb{Z}^N.
    \]
\begin{corollary}
\label{corocases}
    Let $r > p^*$ and $r \geqslant q$. If $a$ is T-periodic or bounded potential, then for each $t > 0$, the problem \eqref{eqconstrained} (or \eqref{eq2}) has a positive solution $u$.
\end{corollary}

 Theorem \ref{theoremmain} and Corollary \ref{corocases} can be extended to
quasi-transitive graphs. Let $G = (V,E)$ be a quasi-transitive graph
satisfying $d$-isoperimetric inequality for some $d \geqslant 2$ with $\mu \equiv 1$ and $w \equiv 1$. On one hand, if we replace ${x_n-y_n}$ with ${g_n} \subset Aut(G)$ such that $\{g_n(x_n)\}$ bounded in $V$ and $x+y_n$ with $g_n^{-1}(x)$, then Theorem \ref{theoremmain} holds. On the other hand, if we replace $T$-periodicity with $H$-invariance, $\abs{x}$ with $d(x,x_0)$ for some $x_0 \in V$, and $a(x) \leqslant a_\infty$ with $a(x) < a_\infty$, where $H \leqslant Aut(G)$ and the action of $H$ on $G$ has finitely many orbits, then Corollary \ref{corocases} holds. For more details of quasi-transitive graphs, we may refer to \cite{hua2023existence}.

Finally, motivated by \cite{Fan}, we study the asymptotic behaviour of Lagrangian multiplier $\lambda$. If $a(x) \equiv 0$, then \eqref{eq2} boils down to the following $p$-Laplacian problem
\begin{equation}
    \label{eqp}
    -\mathrm{div}(\abs{\nabla u}^{p-2}\nabla u) = \lambda \abs{u}^{r-2}u \quad \text{ in } \mathbb{Z}^N.
\end{equation}
If there exists  $\lambda_0$ such that problem \eqref{eqp} with $\lambda = \lambda_0$ has a positive (or nontrivial) solution $u_0$, then, thanks to the homogeneity, for any $\lambda_1 > 0$, $u_1 = \left(\frac{\lambda_0}{\lambda_1}\right)^\frac{1}{r-p}u_0$ is a solution with $\lambda = \lambda_1$. However, this is not the case when $a(x) \not\equiv 0$ due to the loss of homogeneity.

We assume that the hypotheses of Theorem \ref{theoremcompact} or Corollary \ref{corocases} hold. For each $t > 0$, put
\[
A_t = \left\{u \in D^{1,\mathcal{H}}(\mathbb{Z}^{N}) : u \text{ is a solution of problem } \eqref{eqconstrained} \text{ and } u > 0\right\}.
\]
 We know that, for each $u_t \in A_t$, there is a positive number $\lambda = \lambda(u_t)$, the Lagrangian multiplier associated with $u_t$, such that $(u_t, \lambda(u_t))$ is a solution of problem \eqref{eq2}. For $\lambda = \lambda(u_t)$ we have the following asymptotic result.
\begin{theorem}
\label{theoremeg}
Under the hypotheses of Theorem \ref{theoremcompact} with $r > q$ or Corollary \ref{corocases}, for any $u_t \in A_t$, we have that
\[
    \lambda(u_t) \to 0 \text{ as } t \to +\infty, \text{ and } \,\,\lambda(u_t) \to \infty \text{ as } t \to 0^{+}.
\]
\end{theorem}
The rest of the paper is organized as follows. In Section \ref{sec:preliminaries}, we introduce the Musielak-Orlicz spaces on locally finite graphs and prove some basic facts of double phase operators. In particular, we prove  some
elementary properties of Musielak-Orlicz on lattice graphs and we prove Theorem \ref{theorem1} in Subsection \ref{ZN}. In Section \ref{sectioncompact}, we prove a compact lemma and give the proof of Theorem \ref{theoremcompact}. In Section \ref{sectionmain}, we prove Theorem \ref{theoremmain} and Corollary \ref{corocases} by excluding cases that $u = 0$ and $0 < J(u) <t$. In the final section  \ref{sectioneg}, we prove Theorem \ref{theoremeg}.
\section{Preliminaries}
\label{sec:preliminaries}

In this section, we first introduce some notations on locally finite graphs. Denote $C(V\times V)$ as the set of all real functions defined on $V\times V$. For any $g \in C(V \times V)$ and fixed $x_0 \in V$, assume that $\{y_1,y_2,...,y_k\}$ be all neighbours of $x_0$, we write
\[
    \sum_{y\sim{x_0}}g(x_0, y) := \sum_{\{y\in V \ : \ y x_0 \in E\}}g(x_0, y) = g(x_0,y_1) + g(x_0,y_2) + \cdots +g(x_0,y_k).
\]

Define the difference operator
\begin{equation}
\begin{aligned}
\nabla \colon C(V) &\to C(V \times V)\\
 u &\to (\nabla u)(x,y) := u(y)-u(x)
  \end{aligned}
\end{equation}
for all $(x,y) \in V \times V$.

The divergence operator is denoted by
\begin{equation}
\begin{aligned}
    \mathrm{div} \colon C(V \times V) &\to C(V)\\
    f &\to (\mathrm{div} f)(x):=\frac{1}{2\mu(x)}\sum_{y\sim{x}}w_{xy}(f(x,y)-f(y,x))
  \end{aligned}
\end{equation}
for all $x \in V$.

\begin{remark}
The operators $\nabla$ and $\mathrm{div}$ are introduced in
  \cite{hua2015time} and \cite{Dei}, and they are employed to establish Proposition \ref{progreen}, which presents Green's formula on graphs.
\end{remark}

The $p$-Laplacian on graphs (see \cite{hua2021existence}) of $u \in C(V)$ can be defined as follows
\begin{equation*}
  \Delta_p u(x) := \mathrm{div}(\abs{\nabla u}^{p-2}\nabla u)(x) =
  \frac{1}{\mu(x)}\sum_{y\sim{x}}w_{xy}\abs{u(y)-u(x)}^{p-2}(u(y)-u(x))
\end{equation*}
for all $x \in V$ with $1 < p < \infty$.

Now we define the double phase operator on graphs of $u\in C(V)$ as follows
\begin{equation}
  \label{eqdoublephase}
  \begin{aligned}
    L(u)(x_0):
    = & -\mathrm{div}(\abs{\nabla u}^{p-2}\nabla u+a(x)\abs{\nabla u}^{q-2}\nabla u)(x_0)\\
    = & -\frac{1}{\mu(x_0)}\sum_{y\sim{x_0}}w_{x_0y}\abs{u(y)-u(x_0)}^{p-2}(u(y)-u(x_0))\\
      & -\frac{1}{\mu(x_0)}\sum_{y\sim{x_0}}w_{x_0y}\frac{a(x_0)+a(y)}{2}\abs{u(y)-u(x_0)}^{q-2}(u(y)-u(x_0))
      \end{aligned}
\end{equation}
for all $x_0 \in V$, $1<p<q<\infty$ and $0 \leqslant a(\cdot) \in C(V)$.

\begin{remark}
  Throughout the paper, if $a \in C(V)$, then we also write $a(x)$ as $a_{V \times V} \in C(V \times V)$ when there is no confusion, where $a_{V \times V}(x_0,y_0):=a(x_0)$ for all $(x_0,y_0) \in V \times V$.
\end{remark}
The integral of $u\in C(V)$ over $V$ is defined by
\[
\int_{V}u\, d\mu := \sum_{x \in V}\mu(x)u(x).
\]
For any $1 \leqslant p \leqslant \infty$, $\ell^p(V)$ denotes the linear space of $p$-the integrable functions on $V$ equipped with the norm
\[
  \norm{u}_{\ell^p(V)}:=
  \begin{cases}
    (\int_{V}\abs{u}^p\, d\mu)^{\frac{1}{p}}, \quad  &1\leqslant p<\infty,\\
    ~ \sup\limits_{x\in V}\abs{u(x)},     &p = \infty.
\end{cases}
\]
Similarly, the integral of $f\in C(V \times V)$ over $V \times V$ is defined by
\[
\int_{V \times V}f\, dw := \frac{1}{2}\sum_{(x,y) \in V \times V}w_{xy}f(x,y).
\]
It follows from the definition of $w$ that
\[
\int_{V \times V}f\, dw =\frac{1}{2}\sum_{xy \in E }w_{xy}f(x,y).
\]
Thus we prefer to work with $\ell^p(E)$, which denotes the linear space of $p$-the integrable functions on $E$ equipped with the norm
\[
  \norm{f}_{\ell^p(E)}:=
  \begin{cases}
    (\int_{E}\abs{f}^p\, dw)^{\frac{1}{p}}, \quad  &1\leqslant p<\infty,\\
    \sup\limits_{xy\in E}\abs{f(x,y)},     &p = \infty,
\end{cases}
\]
for any $1 \leqslant p \leqslant \infty$. Moreover, for $f \in \ell^1(E)$, it is easy to have that
\[
   \int_{E}f\, dw = \frac{1}{2}\sum_{x \in V}\sum_{y\sim{x}} w_{xy}f(x,y).
\]
From a point of view of measure theory, both $(V,2^V,\mu)$ and $(E,2^E,\frac{w}{2}\big|_E)$ are $\sigma$-finite,
complete measure spaces equipped with a separable measure. Thus, we conclude some elementary properties of $\ell^p(V)$ and $\ell^p(E)$, see the details in \cite[Theorem 4.1.3]{bogachev2007measure},\cite[Theorem 4.7.15]{bogachev2007measure} and  \cite[Theorem 3.4.4]{diening2011lebesgue}.
\begin{lemma}
  Let $G=(V,E)$ be a connected locally finite graph, both $\ell^p(V)$ and $\ell^p(E)$ have the following properties:
  \begin{enumerate}[label=(\roman*)]
      \item  they are Banach spaces for $1 \leqslant p \leqslant \infty$;
      \item  they are uniformly convex, and so reflexive for $1 < p < \infty$;
      \item  they are separable for $1 \leqslant p < \infty$.
  \end{enumerate}
\end{lemma}
\subsection{The Musielak-Orlicz space}
\label{subMOspace}
\quad \\
In this subsection, we recall some basic definitions and properties of Musielak-Orlicz space $L^\varphi(A,\mu)$,
where $(A,\Sigma,\mu)$ is a $\sigma$-finite, complete measure space.
For more details, please see \cite{diening2011lebesgue}.
\begin{definition}
  \label{defspace}
  A convex, left-continuous function $\varphi \colon [0, \infty) \to [0, \infty]$
with
\[\varphi(0) = 0,\quad \lim_{t\to 0^+}\varphi(t) = 0
, \quad \text{and} \quad \lim_{t\to \infty}\varphi(t) = \infty
\]
is called a $\Phi$-function.
It is called positive if $\varphi(t) > 0$ for all $t > 0$.

\noindent Let $(A,\Sigma,\mu)$ be a $\sigma$-finite, complete measure space.
  A real function $\varphi \colon A\times[0, \infty) \to [0, \infty]$ is said to be a generalized $\Phi$-function on
$(A,\Sigma,\mu)$, denoted by $\varphi \in \Phi(A,\mu)$. If
\begin{enumerate}[label=(\roman*)]
  \item  $\varphi(y,\cdot)$ is a $\Phi$-function for every $y \in A$.
  \item  $y \to \varphi(y,t)$ is measurable for every $t \geqslant 0$.
\end{enumerate}
Let $\varphi \in \Phi(A,\mu)$ and let $\varrho_\varphi$ be given by
\[
\varrho_\varphi(f) :=\int_A\varphi(y,\abs{f(y)})\, d\mu(y)
\]
for all $\mu$-measurable real functions $f$ on $A$. Then the Musielak-Orlicz space
is defined by
\[
  L^\varphi(A,\mu):=\left\{f \colon A \to \mathbb{R} ~\text{measurable}:
  \varrho_\varphi(\lambda f) < \infty~\text{for some}~\lambda > 0\right\},
\]
equipped with the norm
\[
  \norm{f}_\varphi:=\mathrm{inf}\left\{\lambda>0:
  \varrho_\varphi(\frac{f}{\lambda})\leqslant 1\right\}.
\]
\end{definition}
\begin{proposition}
  \label{probanach}
  Let $\varphi \in \Phi(A,\mu)$, then $L^\varphi(A,\mu)$ is a Banach space.
\end{proposition}
\begin{definition}
  We say that $\varphi \in \Phi(A,\mu)$ satisfies the $\Delta_2$-condition if
there exists $K \geqslant 2$ such that
\[
\varphi(y,2t) \leqslant K\varphi(y,t)
\]
for all $y \in A$ and all $t > 0$. The smallest such $K$ is called the $\Delta_2$-constant
of $\varphi$.

\noindent A $\Phi$-function $\varphi$ is said to be an N-function if it is continuous and positive and satisfies
\[
  \lim_{t \to 0} {\frac{\varphi(t)}{t}}=0 \quad \text{and} \quad \lim_{t \to \infty}{\frac{\varphi(t)}{t}}=\infty.
\]

\noindent A function $\varphi \in \Phi(A,\mu)$ is said to be a generalized N-function if
$\varphi(y,\cdot)$ is an N-function for every $y \in A$, denoted by $\varphi \in N(A,\mu)$.

\noindent A function $\varphi \in N(A,\mu)$ is called uniformly convex if for any
$\varepsilon > 0$ there exists $\delta > 0$ such that
\[
\abs{t_1-t_2} \leqslant \varepsilon \max{\{t_1, t_2\}} \quad \text{or} \quad \varphi\left(y,\frac{t_1+t_2}{2}\right)\leqslant(1-\delta)
\frac{\varphi(y,t_1)+\varphi(y,t_2)}{2}
\]
for all $t_1, t_2 \geqslant 0$ and every $y \in A$.
\end{definition}
\begin{proposition}
  \label{prouniconvex}
  If $\varphi \in N(A,\mu)$ is uniformly convex and satisfies the $\Delta_2$-condition,
   then the norm $\norm{\cdot}_\varphi$ is uniformly convex. Hence, $L^\varphi(A,\mu)$ is also uniformly convex.
\end{proposition}
\begin{definition}
A measure $\mu$ is called separable if there exists a sequence
$\{E_k\} \subset \Sigma$ with the following properties:
\begin{enumerate}[label=(\roman*)]
  \item $\mu(E_k) < \infty$ for all $k \in \mathbb{N}^+$.
  \item  for every $E \in \Sigma$ with $\mu(E) < \infty$ and every $\varepsilon > 0$ there exists an index $k$ such that $\mu(E \triangle E_k) < \varepsilon$, where $\triangle$ denotes the symmetric difference defined through $E \triangle E_k :=  (E \setminus E_k) \cup (E_k \setminus E)$.
\end{enumerate}

A function $\varphi \in \Phi(A,\mu)$ is called locally integrable on $A$ if
  \[
  \varrho_\varphi(t\chi_E) < \infty
  \]
  for all $t > 0$ and all $\mu$-measurable $E \subset A$ with $\mu(E) < \infty$.
\end{definition}
\begin{proposition}
  \label{proequal}
If $\varphi \in \Phi(A,\mu)$ satisfies the $\Delta_2$-condition, then
\[
\begin{aligned}
      L^\varphi(A,\mu)
    & =E^\varphi(A,\mu):= \left\{f \in L^\varphi(A,\mu) : \varrho_\varphi(\lambda f) < \infty ~\text{for all}~ \lambda > 0\right\}\\
     & =L_{OC}^\varphi(A,\mu):=\left\{f \in L^\varphi(A,\mu) : \rho_\varphi(f) < \infty \right\}.
  \end{aligned}
 \]
\end{proposition}
\begin{proposition}
  \label{proseparable}
  Let $\varphi \in \Phi(A,\mu)$ be locally integrable and let $\mu$ be separable, then
  $E^\varphi(A,\mu)$ is separable. If we further assume
  that $\varphi \in \Phi(A,\mu)$ satisfies the $\Delta_2$-condition, then
  $L^\varphi(A,\mu)$ is separable by Proposition \ref{proequal}.
\end{proposition}
\subsection{The spaces \texorpdfstring{$\ell^{\mathcal{H}_V}(V)$}{Lg} and \texorpdfstring{$\ell^{\mathcal{H}_E}(E)$}{Lg}}
\label{subsectionHVE}
\quad \\
Let $\mathcal{H}_V \colon V \times [0,\infty) \to [0,\infty)$ be defined by
\[
\mathcal{H}_V(x,t):= t^p+a(x)t^q
\]
for all $x \in V$ and $t \geqslant 0$, with $1<p<q<\infty$ and $0 \leqslant a(\cdot) \in C(V)$.

Let $\mathcal{H}_E \colon E \times [0,\infty) \to [0,\infty)$ be defined by
\[
\mathcal{H}_E((x,y),t):= t^p+a(x)t^q
\]
for all $xy \in E$ and $t \geqslant 0$.
Obviously, $\mathcal{H}_V \in N(V,\mu)$ and $\mathcal{H}_E \in N(E,w)$ satisfy $\Delta_2$-condition.

As in Definition \ref{defspace}, we define the Musielak-Orlicz spaces
\[
  \ell^{\mathcal{H}_V}(V):= L^{\mathcal{H}_V}(V,\mu) =\left\{u \colon V \to \mathbb{R} : \rho_{\mathcal{H}_V}(u) < \infty \right\}
\]
and
\[
  \ell^{\mathcal{H}_E}(E):= L^{\mathcal{H}_E}(E,w) =\left\{f \colon E \to \mathbb{R} : \rho_{\mathcal{H}_E}(f) < \infty \right\}
\]
equipped with the norm
\[
  \norm{u}_{\ell^{\mathcal{H}_V}(V)}:=\mathrm{inf}\left\{\lambda>0:
  \rho_{\mathcal{H}_V}(\frac{u}{\lambda})\leqslant 1\right\}
\quad \text{and} \quad
  \norm{f}_{\ell^{\mathcal{H}_E}(E)}:=\mathrm{inf}\left\{\lambda>0:
  \rho_{\mathcal{H}_E}(\frac{f}{\lambda})\leqslant 1\right\}
\]
respectively,
where
\[
  \rho_{\mathcal{H}_V}(u)=\int_V{\left(\abs{u}^p+a(x)\abs{u}^q\right)}\, d\mu
  \quad \text{and} \quad
  \rho_{\mathcal{H}_E}(f)=\int_E{\left(\abs{f}^p+a(x)\abs{f}^q\right)}\, dw
\]
are called $\mathcal{H}_V$-modular and $\mathcal{H}_E$-modular.

Moreover, we have the following two elementary properties of the spaces and modulars.
Their proofs parallel those of the Euclidean case and one can refer to \cite[Proposition 2.14]{colasuonno2016eigenvalues} and
\cite[Proposition 2.1]{liu2018existence}. Here, we present their details for completeness of this paper.
\begin{proposition}
\label{proplve}
  Spaces $\ell^{\mathcal{H}_V}(V)$ and $\ell^{\mathcal{H}_E}(E)$ are uniformly convex,
  and so reflexive, Banach spaces. If $a$ is integrable on all $V_1 \subset V$ with $\mu(V_1) < \infty$, then  $\ell^{\mathcal{H}_V}(V)$ is separable. If $a(x)$ is integrable on all $E_1 \subset E$ with $w(E_1) < \infty$, then $\ell^{\mathcal{H}_E}(E)$ is separable.
\end{proposition}
\begin{proof}
By Proposition \ref{probanach}, we know that $\ell^{\mathcal{H}_V}(V)$ and $\ell^{\mathcal{H}_E}(E)$ are Banach spaces.

Next, we show that $\ell^{\mathcal{H}_V}(V)$ and $\ell^{\mathcal{H}_E}(E)$ are uniformly convex. In view of Proposition \ref{prouniconvex}, it suffices to  show that $\mathcal{H}_V$ and $\mathcal{H}_E$ are uniformly convex. By \cite[Remark 2.4.6]{diening2011lebesgue}, $t^p$ and $t^q$ are uniformly convex.
  Therefore, let $\varepsilon > 0$ and $t_1,t_2 \geqslant 0$ be such that $\abs{t_1-t_2} > \varepsilon \max\{t_1,t_2\}$, there
  exist $\delta_p(\varepsilon), \delta_q(\varepsilon)$ such that
  \[
    \left(\frac{t_1+t_2}{2}\right)^p \leqslant 1-\delta_p(\varepsilon) \frac{t_1^p+t_2^p}{2}
    \quad \text{and} \quad
    \left(\frac{t_1+t_2}{2}\right)^q \leqslant 1-\delta_p(\varepsilon) \frac{t_1^q+t_2^q}{2},
  \]
thus
  \[
    \left(\frac{t_1+t_2}{2}\right)^p+a(x)\left(\frac{t_1+t_2}{2}\right)^q
     \leqslant (1-\min\{\delta_p(\varepsilon),\delta_q(\varepsilon)\})        \frac{t_1^p+a(x)t_1^q+t_2^p+a(x)t_2^q}{2}
   \]
  This concludes the uniform convexity of $\mathcal{H}_V$ and $\mathcal{H}_E$.

    Now, suppose that $a$ is integrable on all $V_1 \subset V$ with $\mu(V_1) < \infty$. It is clear that $\mu$ is separable, and $\mathcal{H}_V$ is locally integrable and satisfies $\Delta_2$-condition. Thus, by Proposition \ref{proseparable}, we know that $\ell^{\mathcal{H}_V}(V)$ is separable. Through the similar arguments, we also can show that $\ell^{\mathcal{H}_E}(E)$ is separable, if $a$ is integrable on all $E_1 \subset E$ with $w(E_1) < \infty$.
\end{proof}

\begin{proposition}
\label{promodular}
The $\mathcal{H}_V$-modular has the following properties
\begin{enumerate}[label=(\roman*)]
  \item for $u\neq0$, $\norm{u}_{\mathcal{H}_V}=a \Leftrightarrow \varrho_{\mathcal{H}_V}(\frac{u}{a})=1$;
  \item $\norm{u}_{\mathcal{H}_V}<1$(resp. $=1;>1$) $\Leftrightarrow \varrho_{\mathcal{H}_V}(u)<1$ (resp. $=1;>1$);
  \item $\norm{u}_{\mathcal{H}_V}<1 \Rightarrow \norm{u}_{\mathcal{H}_V}^q \leqslant \varrho_{\mathcal{H}_V}(u) \leqslant \norm{u}_{\mathcal{H}_V}^p$;
        $\norm{u}_{\mathcal{H}_V}>1 \Rightarrow \norm{u}_{\mathcal{H}_V}^p \leqslant \varrho_{\mathcal{H}_V}(u) \leqslant \norm{u}_{\mathcal{H}_V}^q$;
  \item $\norm{u}_{\mathcal{H}_V} \to 0 \Leftrightarrow \varrho_{\mathcal{H}_V}(u) \to 0$;
        $\norm{u}_{\mathcal{H}_V} \to \infty \Leftrightarrow \varrho_{\mathcal{H}_V}(u) \to \infty$.
\end{enumerate}
For $\mathcal{H}_E$-modular, we have similar properties.
\end{proposition}
\begin{proof}
For any $u \in \ell^{\mathcal{H}_V}(V)$, it is easy to check that $ \rho_{\mathcal{H}_V}(au)$ is a continuous convex even function in $a$, and strictly increases on $a \in[0,+\infty)$. Thus, by the definition of $\rho_{\mathcal{H}_V}$, we have that
$$
\|u\|_{\mathcal{H}_V}=a \Leftrightarrow \rho_{\mathcal{H}_V}\left(\frac{u}{a}\right)=1,
$$
that is, (i) follows. Therefore, (ii) also follows. Noting that, for any $u \in \ell^{\mathcal{H}_V}(V)$, we have that
\begin{equation}\label{eqmodularpro}
    \begin{aligned}
& b^p \rho_{\mathcal{H}_V}(u) \leqslant \rho_{\mathcal{H}_V}(b u) \leqslant b^q \rho_{\mathcal{H}_V}(u) \text { if } b>1, \\
& b^q \rho_{\mathcal{H}_V}(u) \leqslant \rho_{\mathcal{H}_V}(b u) \leqslant b^p \rho_{\mathcal{H}_V}(u) \text { if } 0<b<1 .
\end{aligned}
\end{equation}
Let $\|u\|_{\mathcal{H}_V}=a$ with $0<a<1$, by (i), we have $\rho_{\mathcal{H}_V}(\frac{u}{a})=1$. Noting that $\frac{1}{a}>1$, by \eqref{eqmodularpro}, we have that
$$
\frac{\rho_{\mathcal{H}_V}(u)}{a^p} \leqslant \rho_{\mathcal{H}_V}\left(\frac{u}{a}\right)=1 \leqslant \frac{\rho_{\mathcal{H}_V}(u)}{a^q},
$$
thus the first part of (iii) follows, and the second part of (iii) is similar. From (iii), it is easy to obtain (iv).
The same argument also works for $\mathcal{H}_E$-modular.
\end{proof}

\subsection{The spaces \texorpdfstring{$D^{1,\mathcal{H}}(V)$}{Lg} and double phase operator}  \quad \\
Let $V$ be a infinite set. In order to look for the solutions of problem \eqref{eq1}, we shall
first introduce work space $D^{1,\mathcal{H}}(V)$ and prove some properties of double phase operator on locally finite graphs.

Let $C_0(V)$ denotes the set of all real functions with finite support,
$D^{1,\alpha}(V)$ denotes the completion of $C_0(V)$ under the norm
\[
\norm{u}_{D^{1,\alpha}(V)} := \norm{\nabla u}_{\ell^\alpha(E)}
\]
for $1 \leqslant \alpha \leqslant \infty$, and $D^{1,\mathcal{H}}(V)$ denotes the completion of $C_0(V)$ under the norm
\[
  \norm{u}_{D^{1,\mathcal{H}}(V)}:=\norm{\nabla u}_{\ell^{\mathcal{H}_E}(E)}.
\]

From now on, we abbreviate
$\norm{u}_{D^{1,\mathcal{H}}(V)}$ to $\norm{u}$, $\varrho_{\mathcal{H}_E}\left(\nabla u\right)$ to $\varrho\left(\nabla u\right)$, $\norm{u}_{D^{1,\alpha}(V)}$ to $\norm{\nabla u}_\alpha$, and
$\norm{u}_{\ell^\alpha(V)}$ to $\norm{u}_\alpha$.

\begin{proposition}
  \label{prospaceD1V}
  $D^{1,\mathcal{H}}(V)$ is separable, uniformly convex,
  and so reflexive, Banach space.
\end{proposition}
\begin{proof}
   It is clear that $\nabla \colon D^{1,\mathcal{H}}(V) \to \ell^{\mathcal{H}}(E)$ is an isometry. Thus, Proposition \ref{proplve} and \cite[Theorem 1.22]{sobolev} imply that $D^{1,\mathcal{H}}(V)$ is a uniformly convex, and so reflexive, Banach space. Since $C_0(V)$ is dense in $D^{1,\mathcal{H}}(V)$, we know that $\left\{u \in C_0(V) : u(x) \in \mathbb{Q} \text{ for all } x \in V\right\}$ is dense in $D^{1,\mathcal{H}}(V)$. Thus, $D^{1,\mathcal{H}}(V)$ is separable.
\end{proof}
Define
\begin{equation}
\begin{aligned}
\label{eqI}
I \colon D^{1,\mathcal{H}}(\mathbb{Z}^N) &\to \mathbb{R}\\
 u &\to \int_E{\left(\frac{1}{p}\abs{\nabla u}^p + \frac{a(x)}{q}\abs{\nabla u}^q\right)}\, dw,
\end{aligned}
\end{equation}
It is clear that $I \in C^1(D^{1,\mathcal{H}}(V),\mathbb{R})$ and double phase operator $L$, defined by \eqref{eqdoublephase}, is the derivative operator of $I$
  in the weak sense, that is, $L = I' \colon  D^{1,\mathcal{H}}(V) \to \left(D^{1,\mathcal{H}}(V)\right)^*$, i.e.,
  \[
  \left \langle L(u),v \right \rangle :=
  \int_E{\abs{\nabla u}^{p-2}{\nabla u}{\nabla v}}\, dw
   + \int_E{a(x)\abs{\nabla u}^{q-2}{\nabla u}{\nabla v}}\, dw
\]
for all $u,v \in D^{1,\mathcal{H}}(V)$. Here $\left(D^{1,\mathcal{H}}(V)\right)^*$ denotes the dual space of $D^{1,\mathcal{H}}(V)$ and $\left \langle \cdot,\cdot \right \rangle$ denotes the pairing between $D^{1,\mathcal{H}}(V)$ and $\left(D^{1,\mathcal{H}}(V)\right)^*$.

By adapting the arguments of \cite[Theorem 3.1]{Fan2} and \cite[Proposition 3.1]{liu2018existence}, we have the following elementary properties of double phase operator $L$. From now on, for simplicity, we write by $C$ the generic positive constants (the exact value may be different from line to line and
even different in the same line).
\begin{proposition}
  The double phase operator $L$ has the following properties.
  \begin{enumerate}[label=(\roman*)]
    \item $L$ is a continuous, bounded and strictly monotone operator;
    \item $L$ is a mapping of type $(S)_+$, i.e., if $u_n \rightharpoonup u$ and
    $\varlimsup\limits_{n \to \infty}{\left \langle L(u_n)-L(u),u_n-u \right \rangle} \leqslant 0$,
    then $u_n \to u$ in $D^{1,\mathcal{H}(V)}$;
    \item $L$ is a homeomorphism.
  \end{enumerate}
\end{proposition}
\begin{proof}
    (i) It is clear that $L: D^{1,\mathcal{H}}(V) \to \left(D^{1,\mathcal{H}}(V)\right)^*$ is  continuous. Now we prove that $L$ is bounded. For convenience in writing we set $\lambda_{1}:=\norm{u}, \lambda_{2}:=\norm{v}$. By H\"{o}lder's inequality, we have that
\[
\begin{aligned}
\left|\left\langle L(u), v \right\rangle\right| = & \left| \int_{E}\abs{\nabla u}^{p-2} \nabla u \nabla v\, dw + \int_{E} a(x)\abs{\nabla u}^{q-2} \nabla u \nabla v \, dw \right|\\
\leqslant & \left(\int_{E}\abs{\nabla u}^p dw\right)^{\frac{p-1}{p}}\left(\int_{E}\abs{\nabla v}^p dw\right)^{\frac{1}{p}} \\
& +\left(\int_{E} a(x)\abs{\nabla u}^{q} dw\right)^{\frac{q-1}{q}}\left(\int_{E} a(x) \abs{\nabla v}^{q} dw\right)^{\frac{1}{q}}.
\end{aligned}
\]
Hence, by Proposition \ref{eqmono}, we have that
\[
\norm{L(u)}_{\left(D^{1,\mathcal{H}}(V)\right)^{*}}=\sup _{\norm{v} \leqslant 1}|\langle L(u), v\rangle| \leqslant \left(\varrho\left( \nabla u \right) \right)^\frac{p-1}{p} +  \left(\varrho\left( \nabla u \right)\right)^\frac{q-1}{q},
\]
which implies that $L$ is bounded.

The strict monotonicity of $L$ follows easily from the following inequalities, and we refer to \cite{Kichenassamy1986Singular,Lindqvist2017notes},
\begin{equation}
\label{eqmono}
\begin{aligned}
& \left(|\xi|^{p-2} \xi-|\eta|^{p-2} \eta\right)(\xi-\eta) \cdot\left(|\xi|^{p}+|\eta|^{p}\right)^{\frac{2-p}{p}} \geqslant(p-1)|\xi-\eta|^{p} \quad \text { if } 1<p<2, \\
& \left(|\xi|^{p-2} \xi-|\eta|^{p-2} \eta\right)(\xi-\eta) \geqslant\left(\frac{1}{2}\right)^{p}|\xi-\eta|^{p} \quad \text { if } p \geqslant 2
\end{aligned}
\end{equation}
(ii) Assume that $\left\{u_{n}\right\} \subset D^{1,\mathcal{H}}(V)$, $u_{n}  \rightharpoonup u$ and
\[
\varlimsup _{n \to\infty}\left\langle L\left(u_{n}\right)-L(u), u_{n}-u\right\rangle \leqslant 0 .
\]
Then, since $L$ is monotone, we have
\[
\varliminf _{n \to\infty}\left\langle L\left(u_{n}\right)-L(u), u_{n}-u\right\rangle \geqslant 0.
\]
Thus
\[
\lim _{n \to\infty}\left\langle L\left(u_{n}\right)-L(u), u_{n}-u\right\rangle=0,
\]
that is,
\[
\begin{aligned}
\lim _{n \to \infty} & \left(\int_E\left(\left|\nabla u_{n}\right|^{p-2} \nabla u_{n}-|\nabla u|^{p-2} \nabla u\right)\left(\nabla u_{n}-\nabla u\right)\, dw\right. \\
& \left.+\int_E a(x)\left(\left|\nabla u_{n}\right|^{q-2} \nabla u_{n}-|\nabla u|^{q-2} \nabla u\right)\left(\nabla u_{n}-\nabla u\right)\, dw\right)=0.
\end{aligned}
\]

Next, we claim that $\nabla u_{n} \to \nabla u$ pointwise in $E$. In fact, it follows from $u_{n} \rightharpoonup u$ and Proposition \ref{promodular}(iv) that $\{\varrho(\nabla u_n)\}$ is bounded. Then,
\[
    \abs{\nabla u_n}^p(x,y)w_{xy} \leqslant \int_E{\abs{\nabla u_n}^p}\, dw \leqslant \varrho(\nabla u_n)
\]
shows that $\{ \abs{\nabla u_n}^p(x,y) \}$ is bounded for every $xy \in E$. For $1 < p < 2$, by \eqref{eqmono}, we have
\begin{equation*}
\begin{aligned}
(p-1)\abs{\nabla (u_n - u)}^p(x,y)  &\leqslant \left(\left(\left|\nabla u_{n}\right|^{p-2} \nabla u_{n}-|\nabla u|^{p-2} \nabla u\right)\left(\nabla u_{n}-\nabla u\right)\cdot{\left(\abs{\nabla u_n}^p + \abs{\nabla u}^p\right)}^\frac{2-p}{p}\right)(x,y)\\
&\leqslant \frac{1}{w_{xy}}\left\langle L\left(u_{n}\right)-L(u), u_{n}-u\right\rangle {\left(\abs{\nabla u_n}^p + \abs{\nabla u}^p\right)}^\frac{2-p}{p}(x,y) \to 0
\end{aligned}
\end{equation*}
for each $xy \in E$. Thus, we conclude $\nabla u_{n} \to \nabla u$ pointwise in $E$. The claim is trivial for $p \geqslant 2$ by \eqref{eqmono}.

By H\"{o}lder's inequality and Young's inequality, it is easy to compute  that
\[
\begin{aligned}
\left\langle L\left(u_{n}\right), u_{n}-u\right\rangle= & \int_{E}\left(\left|\nabla u_{n}\right|^{p-2} \nabla u_{n}\right)\left(\nabla u_{n}-\nabla u\right) dw \\
& +\int_{E} a(x)\left(\left|\nabla u_{n}\right|^{q-2} \nabla u_{n}\right)\left(\nabla u_{n}-\nabla u\right) dw \\
= & \int_{E}\left(\left|\nabla u_{n}\right|^{p}+a(x)\left|\nabla u_{n}\right|^{q}\right) dw \\
& -\int_{E}\left(\left|\nabla u_{n}\right|^{p-2} \nabla u_{n} \nabla u+a(x)\left|\nabla u_{n}\right|^{q-2} \nabla u_{n} \nabla u\right)\, dw \\
\geqslant & \int_{E}\left(\left|\nabla u_{n}\right|^{p}+a(x)\left|\nabla u_{n}\right|^{q}\right) dw-\int_{E}\left(\left|\nabla u_{n}\right|^{p-1}|\nabla u|+a(x)\left|\nabla u_{n}\right|^{q-1}|\nabla u|\right)\, dw \\
\geqslant & \int_{E}\left(\left|\nabla u_{n}\right|^{p}+a(x)\left|\nabla u_{n}\right|^{q}\right) dw-\frac{p-1}{p} \int_{E}\left|\nabla u_{n}\right|^{p}\, dw-\frac{1}{p} \int_{E}|\nabla u|^{p}\, dw \\
& -\frac{q-1}{q} \int_{E} a(x)\left|\nabla u_{n}\right|^{q}\, dw-\frac{1}{q} \int_{E} a(x)|\nabla u|^{q}\, dw \\
= & \int_{E}\left(\frac{1}{p}\left|\nabla u_{n}\right|^{p}+\frac{1}{q} a(x)\left|\nabla u_{n}\right|^{q}\right) dw-\int_{E}\left(\frac{1}{p}|\nabla u|^{p}+\frac{1}{q} a(x)|\nabla u|^{q}\right) dw .
\end{aligned}
\]
Noting that
\[
\lim _{n \to\infty}\left\langle L(u), u_{n}-u\right\rangle=0,
\]
we conclude that
\[
\lim_{n \to\infty}\left\langle L\left(u_{n}\right), u_{n}-u\right\rangle = 0,
\]
that is,
\[
\varlimsup _{n \to\infty} \int_{E}\left(\frac{1}{p}\left|\nabla u_{n}\right|^{p}+\frac{1}{q} a(x)\left|\nabla u_{n}\right|^{q}\right) dw \leqslant \int_{E}\left(\frac{1}{p}|\nabla u|^{p}+\frac{1}{q} a(x)|\nabla u|^{q}\right) dw .
\]
On the other hand, it follows from Fatou's lemma that
\[
\varliminf _{n \to\infty} \int_{E}\left(\frac{1}{p}\left|\nabla u_{n}\right|^{p}+\frac{1}{q} a(x)\left|\nabla u_{n}\right|^{q}\right) dw \geqslant \int_{E}\left(\frac{1}{p}|\nabla u|^{p}+\frac{1}{q} a(x)|\nabla u|^{q}\right) dw .
\]
Thus, we have that
\[
\lim _{n \to\infty} \int_{E}\left(\frac{1}{p}\left|\nabla u_{n}\right|^{p}+\frac{1}{q} a(x)\left|\nabla u_{n}\right|^{q}\right) dw=\int_{E}\left(\frac{1}{p}|\nabla u|^{p}+\frac{1}{q} a(x)|\nabla u|^{q}\right) dw,
\]
Since
\[
\begin{aligned}
\left|\nabla u_{n}-\nabla u\right|^{p}+a(x)\left|\nabla u_{n}-\nabla u\right|^{q} \leqslant & C\left(\frac{1}{p}\left|\nabla u_{n}\right|^{p}+\frac{1}{q} a(x)\left|\nabla u_{n}\right|^{q}\right) \\
& +C\left(\frac{1}{p}|\nabla u|^{p}+\frac{1}{q} a(x)|\nabla u|^{q}\right),
\end{aligned}
\]
it follows from a variant of Lebesgue dominated convergence theorem that
\[
\lim _{n \to\infty} \int_{E}\left(\left|\nabla u_{n}-\nabla u\right|^{p}+a(x)\left|\nabla u_{n}-\nabla u\right|^{q}\right) dw=0.
\]
Hence, we obtain that $u_{n} \to u$ in $D^{1,\mathcal{H}}(V)$.

(iii) Since $L$ is strictly monotone and
\[
\lim _{\|u\| \to\infty} \frac{\langle L(u), u\rangle}{\|u\|}=\frac{\int_{E}\left(|\nabla u|^{p}+a(x)|\nabla u|^{q}\right)\, dw}{\|u\|}=+\infty,
\]
that is, $L$ is coercive, we know $L$ is a surjection in view of Minty-Browder theorem (see \cite[Theorem 26A]{ZE}). Hence $L$ has an inverse mapping $L^{-1}: \left(D^{1,\mathcal{H}}(V)\right)^* \to D^{1,\mathcal{H}}(V)$. Therefore, in order to complete the proof of (iii), it suffices to prove that $L^{-1}$ is continuous. If $f_{n}, f \in \left(D^{1,\mathcal{H}}(V)\right)^*, f_{n} \to f$, letting $u_{n}=L^{-1}\left(f_{n}\right), u=L^{-1}(f)$, then $L\left(u_{n}\right)=f_{n}, L(u)=f$. Note that $\left\{u_{n}\right\}$ is bounded in $D^{1,\mathcal{H}}(V)$. Therefore, there is a subsequence $\{u_{n_k}\}$ that converges weakly to a limit, say $u_0$. We conclude from $f_{n_k} \to f$ that
\[
\lim _{k \to\infty}\left\langle L\left(u_{n_k}\right)-L\left(u_0\right), u_{n_k}-u_0\right\rangle=\lim _{k \to\infty}\left\langle f_{n_k}, u_{n_k}-u_0\right\rangle=0.
\]
Since $L$ is of type $(S)_+$, $u_{n_k} \to u_0$. Then, by injectivity of $L$, we have $u_0 = u$. Reasoning as above, we know that every subsequence of $\{u_n\}$ contains a convergent subsequence that converges to $u$. Hence, $u_n \to u$,  and $L^{-1}$ is continuous.
\end{proof}
Now, by adapting the argument of \cite[Theorem 2.1]{Gri0}, we derive the integration by parts formula for double phase operator $L$.
\begin{proposition}
  \label{progreen}
Let $u \in D^{1,\mathcal{H}}(V)$ and $v \in C_0(V)$. Then
\begin{equation}
\label{eqgreen}
  \left \langle L(u),v \right \rangle=-\int_V{v \cdot \mathrm{div}(\abs{\nabla u}^{p-2}\nabla u+a(x)\abs{\nabla u}^{q-2}\nabla u)}\, d\mu
\end{equation}
\end{proposition}
\begin{proof}
  Let $f := \abs{\nabla u}^{p-2}\nabla u+a(x)\abs{\nabla u}^{q-2}\nabla u$, then
  \[
    \begin{aligned}
      \left \langle L(u),v \right \rangle
      & = \int_{E}{f(x,y)(\nabla v)(x,y)}\, dw\\
      & = \frac{1}{2}\sum_{xy \in E }w_{xy}{f(x,y)(\nabla v)(x,y)}\\
      & = \frac{1}{2}\sum_{x \in V }\sum_{y \sim x}w_{xy}f(x,y)(v(y)-v(x))\\
      & = \frac{1}{2}\sum_{x \in V }\sum_{y \sim x}w_{xy}f(x,y)v(y)-
      \frac{1}{2}\sum_{x \in V }\sum_{y \sim x}w_{xy}f(x,y)v(x)\\
      & = \frac{1}{2}\sum_{y \in V }\sum_{x \sim y}w_{yx}f(y,x)v(x)-
      \frac{1}{2}\sum_{x \in V }\sum_{y \sim x}w_{xy}f(x,y)v(x)\\
      & = \frac{1}{2}\sum_{yx \in E }w_{yx}f(y,x)v(x)-
      \frac{1}{2}\sum_{x \in V }\sum_{y \sim x}w_{xy}f(x,y)v(x)\\
      & = \frac{1}{2}\sum_{xy \in E }w_{xy}f(y,x)v(x)-
      \frac{1}{2}\sum_{x \in V }\sum_{y \sim x}w_{xy}f(x,y)v(x)\\
      & = - \frac{1}{2}\sum_{x \in V }\sum_{y \sim x}w_{xy}(f(x,y)-f(y,x))v(x)\\
      & = - \int_V{v \cdot (\mathrm{div}f)(x)}\, d\mu\\
      & = -\int_V{v \cdot \mathrm{div}(\abs{\nabla u}^{p-2}\nabla u+a(x)\abs{\nabla u}^{q-2}\nabla u)}\, d\mu.
    \end{aligned}
  \]
\end{proof}
\subsection{The space \texorpdfstring{$D^{1,\mathcal{H}}(\mathbb{Z}^N)$}{Lg} and proof of Theorem \ref{theorem1}}  \quad \\
\label{ZN}
From now on, we study problem \eqref{eq1} on lattice graph $\mathbb{Z}^{N}$. Recalling that $\mu \equiv 1$ and $w \equiv 1$ on lattice graph $\mathbb{Z}^{N}$, we know that $I$ ( and $L$) defined in Section \ref{secin} and \ref{sec:preliminaries} coincide.

The following  interpolation inequalities are crucial for below, and one can refer to \cite[Lemma 2.1]{Huang} for the details.
\begin{proposition}
  \label{prointer}
Suppose $u\in \ell^{\alpha}(\mathbb{Z}^{N})$, then $\norm{u}_\beta \leqslant \norm{u}_\alpha$ for all $1 \leqslant \alpha \leqslant \beta \leqslant \infty$. Moreover, suppose $u \in D^{1,\alpha}(\mathbb{Z}^{N})$, then $\norm{\nabla u}_\beta \leqslant \norm{\nabla u}_\alpha$ for all $1 \leqslant \alpha \leqslant \beta \leqslant \infty$.
\end{proposition}
\begin{proof}
    Let $u\in \ell^{\alpha}(\mathbb{Z}^{N})$ and $\beta \in \left(\alpha, \infty \right)$, we have
    \[
        \norm{u}_\infty \leqslant \norm{u}_\alpha
    \]
    and
    \[
        \norm{u}_\beta^\beta \leqslant \norm{u}_\alpha^\alpha\norm{u}_\infty^{\beta-\alpha},
    \]
  thus
    \[
        \norm{u}_\beta \leqslant \norm{u}_\alpha.
    \]
    Similarity, we may show that $\norm{\nabla u}_\beta \leqslant \norm{\nabla u}_\alpha$  with $u \in D^{1,\alpha}(\mathbb{Z}^{N})$ and for all $1 \leqslant \alpha \leqslant \beta \leqslant \infty$.
\end{proof}
Now, by combining \red{the} Sobolev inequality \eqref{eqsobolev} with Proposition \ref{prointer}, we can conclude the following embedding result.
\begin{proposition}
  \label{proembed}
    Let $N \geqslant 2$, $1 < p < N$ and $ p^*:= \frac{Np}{N-p}$. Then for all $1 \leqslant \alpha \leqslant \beta \leqslant \infty$, the following embeddings hold:
  \begin{enumerate}[label=(\roman*)]
    \item $\ell^\alpha(\mathbb{Z}^N) \hookrightarrow \ell^\beta(\mathbb{Z}^N)$;
    \item $D^{1,\alpha}(\mathbb{Z}^N) \hookrightarrow D^{1,\beta}(\mathbb{Z}^N)$,
    \item   $D^{1,\mathcal{H}}(\mathbb{Z}^N) \hookrightarrow D^{1,p}(\mathbb{Z}^N)$,
    \item   $D^{1,p}(\mathbb{Z}^N) \hookrightarrow \ell^{p^*}(\mathbb{Z}^N)$,
    \item   $D^{1,\mathcal{H}}(\mathbb{Z}^N) \hookrightarrow \ell^{\beta}(\mathbb{Z}^N)$ for all $p^* \leqslant \beta \leqslant \infty$.
  \end{enumerate}
\end{proposition}

\begin{proof}
    The embeddings (i)
 and (ii) follow from Proposition \ref{prointer}.
    For any $u \in C_0(\mathbb{Z}^N) \setminus \{0\}$, we have
    \[
     1 = \varrho\left(\frac{\nabla u}{\norm{\nabla u}}\right) \geqslant \frac{\norm{\nabla u}_p^p}{\norm{\nabla u}^p},
    \]
    which implies the embedding (iii). The embedding (iv) follows from the Sobolev inequality \eqref{eqsobolev} and embedding (v) follows from (i) and (iv).
\end{proof}
\begin{remark}
    \label{remarkequi}
    If $a \in \ell^\infty(\mathbb{Z}^N)$, then we have
    \[
     \varrho\left(\frac{\nabla u}{\norm{\nabla u}}_p\right) \leqslant 1 + \norm{a}_\infty
    \]
    for all $u \in C_0(\mathbb{Z}^N) \setminus \{0\}$. Thus $\norm{u}$ and $\norm{\nabla u}_p$ are equivalent on $C_0(\mathbb{Z}^N)$ by Proposition \ref{proembed} (iii) and  Proposition \ref{promodular} (iv). It implies $D^{1,\mathcal{H}}(\mathbb{Z}^N)$ and $D^{1,p}(\mathbb{Z}^N)$ are same space up to equivalence of norms.
\end{remark}
It is the position to give the proof of Theorem \ref{theorem1}.
\begin{proof}[Proof of Theorem 1.1]
By H\"{o}lder's inequality, Proposition \ref{proembed} (i) and (v), we have that
\[
    \langle f, v\rangle=\int_{\mathbb{Z}^N} f(x) v d\mu \leqslant \norm{f}_{\frac{p^*}{p^*-1}}\norm{v}_{p^*} \leqslant \norm{f}_{r}\norm{v}_{p^*}
\]
for all $v \in D^{1,\mathcal{H}}(\mathbb{Z}^N)$.
Thus $f \in (D^{1,\mathcal{H}}(\mathbb{Z}^N))^*$. Since $L$ is a homeomorphism, there exists a unique $u \in D^{1,\mathcal{H}}(\mathbb{Z}^N)$ such that
\[
     \int_E{\abs{\nabla u}^{p-2}{\nabla u}{\nabla v}}\, dw
    + \int_E{a(x)\abs{\nabla u}^{q-2}{\nabla u}{\nabla v}}\, dw
     = \int_{\mathbb{Z}^N} f(x)v\, d\mu
\]
for all $v \in D^{1,\mathcal{H}}(\mathbb{Z}^N)$.
\end{proof}
\section{Proof of Theorem \ref{theoremcompact}}\label{sectioncompact}
In this section, by adapting the arguments of \cite[Lemma 2.2]{Gri1}, we first prove a compact lemma, which is useful to the proof of main theorem.
\begin{lemma}
\label{lemmacompact}
    Assume that $a$ satisfies
 \[
    \lim_{\abs{x} \to \infty}a(x) = +\infty.
 \]
 Then, $D^{1,\mathcal{H}}(\mathbb{Z}^N)$ is compactly embedded in $\ell^{r}(\mathbb{Z}^N)$ for all $r > p^*$.
\end{lemma}
\begin{proof}
    Let $\{u_n\}$be a bounded sequence in $D^{1,\mathcal{H}}(\mathbb{Z}^N)$. It follows from Proposition \ref{proembed} (v) that $\{\norm{u_n}_\infty\}$ is bounded. Moreover, passing to a subsequence if necessary, we obtain $u_n \to u$ pointwise in $\mathbb{Z}^N$ by Bolzano-Weierstrass theorem and diagonal principle, and thus $\nabla u_n \to \nabla u$ pointwise in $E$.

    Let $v_n=u_n-u$, then $\nabla v_n \to 0$ pointwise in $E$, $\{\norm{\nabla v_n}_p\}$ is bounded, and $\int_E{a(x)\abs{\nabla v_n}^q}\, dw < C$ for all $n \in \mathbb{N}^+$ and some positive constant $C$. Now we need to show that $v_n \to 0$ in $\ell^r(\mathbb{Z}^N)$ for all $r > p^*$.

    Define $B_R := \left\{x \in \mathbb{Z}^N : \abs{x} < R \right\}$ for any $R \geqslant 1$. Since $a$ satisfies $\lim_{\abs{x} \to \infty}a(x) = +\infty$, for any $\varepsilon > 0$, there exists $R_1 \geqslant 1$ such that
    \[
        \frac{C}{\inf\limits_{x \in B_R^c} a(x)} < \frac{\varepsilon}{2},\quad \forall R \geqslant R_1.
    \]

    Fix $R_1 \geqslant 1$, let $R_2 := R_1+1$. Since $E\cap\left(B_{R_2} \times B_{R_2}\right)$ is finite, there exists $N_1 > 0$ such that
    \[
         \int_{E\cap\left(B_{R_2} \times B_{R_2}\right)}{\abs{\nabla v_n}^q}\, dw <  \frac{\varepsilon}{2},\quad \forall n \geqslant N_1.
    \]
   It follows from $E\cap\left(B_{R_2} \times B_{R_2}\right)^c \subset E\cap\left(B_{R_1}^c \times B_{R_1}^c\right)$ that for any $n \geqslant N_1$,
    \[
    \begin{aligned}
         \int_E{\abs{\nabla v_n}^q}\, dw &= \int_{E\cap\left(B_{R_2} \times B_{R_2}\right)}{\abs{\nabla v_n}^q}\, dw + \int_{E\cap\left(B_{R_2} \times B_{R_2}\right)^c}{\abs{\nabla v_n}^q}\, dw\\
         &< \frac{\varepsilon}{2} + \int_{E\cap\left(B_{R_1}^c \times B_{R_1}^c\right)}{\abs{\nabla v_n}^q}\, dw\\
         &< \frac{\varepsilon}{2} + \frac{1}{\inf\limits_{x \in B_{R_1^c}} a(x)}\int_{E\cap\left(B_{R_1}^c \times B_{R_1}^c\right)}{a(x)\abs{\nabla v_n}^q}\, dw\\
        &< \varepsilon.
    \end{aligned}
    \]
    Thus $\norm{\nabla v_n}_q \to 0$ as $n\rightarrow\infty$. Let $s := \frac{Nr}{N+r}$. Then $s^*=r$ and $p < s <N$. Thus, by interpolation inequality and Proposition \ref{prointer}, we have
    \[
        \norm{\nabla v_n}_s^s \leqslant \norm{\nabla v_n}_p^p \norm{\nabla v_n}_\infty^{s-p} \leqslant \norm{\nabla v_n}_p^p \norm{\nabla v_n}_q^{s-p}.
    \]
    Then, by \red{the} Sobolev inequality \eqref{eqsobolev}, we conclude $v_n \to 0$ in $\ell^r(\mathbb{Z}^N)$.
\end{proof}
\begin{proof}[Proof of Theorem \ref{theoremcompact}]
    Let $\{u_n\} \subset D^{1,\mathcal{H}}(\mathbb{Z}^N)$ be a minimizing sequence of problem \eqref{eqconstrained}. Since $\{\abs{v_n}\}$ is also a minimizing sequence, we may assume that $v_n$ are non-negative. \red{Since $r>p^*$, by} Bolzano-Weierstrass theorem, diagonal principle, and Lemma \ref{lemmacompact}, passing to a subsequence if necessary, we obtain $u_n \to u$ pointwise in $\mathbb{Z}^N$, $\nabla u_n \to \nabla u$ pointwise in $E$, and $J(u) = \lim\limits_{n \to \infty}J(u_n) = t$. Moreover, by Fatou's lemma, it implies
    \[
        S_{t} = \varliminf_{n \to \infty}{I(u_n)} \geq I(u) \geq S_{t}.
    \]
    Thus $I(u) = S_{t}$. It is clear that $u$ is a pointwise solution of \eqref{eq2}.

    If $u(x_0)=0$ for some fixed $x \in \mathbb{Z}^N$,
    then \eqref{eq2} yields that
    $u(x) = 0$ for all $x \sim x_0$. Since $\mathbb{Z}^N$ is connected, we obtain $u(x) \equiv 0$, which leads to a contradiction with $J(u) = t$.
\end{proof}
\section{Proof of Theorem \ref{theoremmain}}\label{sectionmain}
Let $a \in \ell^\infty(\mathbb{Z}^N)$. For brevity, we will assume $t = \frac{1}{r}$ and abbreviate $S_{\frac{1}{r}}$ to $S$.
\begin{proposition}\label{prospos}
    $S$ is positive.
\end{proposition}
\begin{proof}
  Since $r > p^*$, by the Sobolev inequality \eqref{eqsobolev} and Proposition \ref{proembed} (iii), we have
  \begin{equation}
    \label{eqsobo}
    \norm{\nabla u}_p \geqslant C_r \norm{u}_r
  \end{equation}
  for all $u \in D^{1,\mathcal{H}}(\mathbb{Z}^N)$, where $C_r$ is a positive constant.

  Hence, we obtain
  \[
    I(\frac{u}{\norm{u}_r})  \geqslant \frac{1}{p}\left(\frac{\norm{\nabla u}_p}{\norm{u}_r}\right)^p \geqslant \frac{C_r^p}{p}
  \]
  for all $u \in D^{1,\mathcal{H}}(\mathbb{Z}^N) \setminus \{0\}$, which, combined with \eqref{eqsobo}, gives
  \[
    S \geqslant \frac{C_r^p}{p} > 0.
  \]
\end{proof}

The following lemma excludes the vanishing case. The proof relies on the arguments in \cite[Lemma 13]{hua2021existence}.
\begin{lemma}
  \label{lemlb}
  Let $\left\{v_n\right\} \subset D^{1,\mathcal{H}}(\mathbb{Z}^N)$
  be a minimizing sequence of \eqref{eqconstrained}. Then $\varliminf\limits _{n\to\infty} \norm{v_n}_{\infty}>0$.
\end{lemma}
\begin{proof}
  Since $r > p^*$, by the Sobolev inequality \eqref{eqsobolev} and Proposition \ref{prointer}, we have
\[
1 = \norm{v_n}_{r}^{r} \leqslant \norm{v_n}_{p^*}^{p^*}\norm{v_n}_\infty^{r-p^*}\leqslant C_{p^*}^{-p^*}\norm{\nabla v_n}_p^{p^*}\norm{v_n}_{\infty}^{r-p^*}
\leqslant C_{p^*}^{-p^*}I(v_n)^\frac{p^*}{p}\norm{v_n}_{\infty}^{r-p^*}.
\]
By taking the limit, we obtain
\[
1 \leqslant {\left(\frac{S}{C_{p^*}^p}\right)}^{\frac{p^*}{p}}\varliminf\limits_{n\to\infty}\norm{v_n}_\infty^{r-p^*},
\]
which, by Proposition \ref{prospos}, implies
\[
\varliminf\limits_{n\to\infty}\norm{v_n}_\infty^{r-p^*}
\geqslant {\left(\frac{C_{p^*}^p}{S}\right)}^{\frac{p^*}{p(r-p^*)}} > 0.
\]
\end{proof}
Now, we give our key lemma to prove Theorem \ref{theoremmain}.
\begin{lemma}
\label{lemkey}
    Let $0 \leqslant a \in \ell^\infty(\mathbb{Z}^N)$. Let $\{v_n\}$ be a minimizing sequence of \eqref{eqconstrained}, $\{y_n\} \subset \mathbb{Z}^N$ be an arbitrary sequence. Then for the sequence after translation $\left\{u_n(x):=v_n(x+y_n)\right\}$ and $\left\{a_n(x):=a(x+y_n)\right\}$, passing to a subsequence if necessary,
    we have
    \begin{equation*}
         \begin{aligned}
             u_n &\to u \text{ pointwise in }\mathbb{Z}^N,\\
             \nabla u_n &\to \nabla u \text{ pointwise in } E,\\
            a_n(x)^\frac{1}{q}\nabla u &\to \tilde{a}(x)^\frac{1}{q}\nabla u \text{ pointwise in }E,\\
        \end{aligned}
    \end{equation*}
$\{ \nabla u_n \}$ is bounded in $\ell^p(E)$, and $\{ a_n(x)^\frac{1}{q}\nabla u_n \}$ is bounded in $\ell^q(E)$. In addition, the following inequality hold:
    \begin{equation}
    \label{eqtrans}
        \lim_{n \to \infty}\int_E{\abs{a_n(x)^\frac{1}{q}\nabla u_n - \tilde{a}(x)^\frac{1}{q}\nabla u}^q\, dw}
         \geqslant
         \varlimsup_{n \to \infty}\int_E{a_n(x)^\frac{1}{q}\abs{\nabla u_n - \nabla u}^q\, dw}.
    \end{equation}
\end{lemma}
\begin{proof}
It follows from $\norm{u_n}_\infty \leqslant \norm{u_n}_r$=$\norm{v_n}_r$ that $\{ \norm{u_n}_\infty \}$ is bounded.
Then passing to a subsequence if necessary, we obtain $u_n \to u$ pointwise in $\mathbb{Z}^N$ by Bolzano-Weierstrass theorem and diagonal principle, and thus $\nabla u_n \to \nabla u$ pointwise in $E$.
Moreover, we obtain $a_n \to \tilde{a}$ pointwise in $\mathbb{Z}^N$ in the same way as above since $a \in \ell^\infty(\mathbb{Z}^N)$. Hence $a_n(x)^\frac{1}{q}\nabla u \to \tilde{a}(x)^\frac{1}{q}\nabla u$ pointwise in $E$.

For the arbitrary sequence $\{y_n\} \subset \mathbb{Z}^N$, the equality
\begin{equation}
    \label{eqIuv}
          \int_E{\left(\frac{1}{p}\abs{\nabla u_n(x)}^p + \frac{a(x+y_n)}{q}\abs{\nabla u_n(x)}^q\right)}\, dw
        =
            \int_E{\left(\frac{1}{p}\abs{\nabla v_n(x)}^p + \frac{a(x)}{q}\abs{\nabla v_n(x)}^q\right)}\, dw
        =I(v_n)
\end{equation}
implies that $\{ \nabla u_n \}$ is bounded in $\ell^p(E)$, and $\{ a_n^\frac{1}{q}\nabla u_n \}$ is bounded in $\ell^q(E)$, since $\{I(v_n)\}$ is bounded. Moreover,  by Fatou's lemma and Proposition \ref{proembed} (ii), we know that $u \in  D^{1,q}(\mathbb{Z}^N)$. By the definition of $a_n$, we get
\begin{equation*}
     \abs{a_n(x)^\frac{1}{q}\nabla u - \tilde{a}(x)^\frac{1}{q}\nabla u} \leqslant 2\norm{a}_\infty^\frac{1}{q}\abs{\nabla u}.
\end{equation*}
Hence, by Lebesgue dominated convergence theorem, we have
\begin{equation}
\label{eqdct}
     \lim_{n \to \infty} \int_E{\abs{a_n(x)^\frac{1}{q}\nabla u - \tilde{a}(x)^\frac{1}{q}\nabla u}^q\, dw}
    =0.
\end{equation}
Finally, by Minkowski inequality and \eqref{eqdct}, passing to a subsequence if necessary, we obtain
 \begin{equation*}
     \begin{aligned}
         \lim_{n \to \infty}\left(\int_E{\abs{a_n(x)^\frac{1}{q}\nabla u_n - \tilde{a}(x)^\frac{1}{q}\nabla u}^q\, dw}\right)^\frac{1}{q}
         \geqslant& \varlimsup_{n \to \infty}\left(\int_E{\abs{a_n(x)^\frac{1}{q}\nabla u_n - a_n(x)^\frac{1}{q}\nabla u}^q\, dw}\right)^\frac{1}{q}\\
         &- \lim_{n \to \infty}\left(\int_E{\abs{a_n(x)^\frac{1}{q}\nabla u - \tilde{a}(x)^\frac{1}{q}\nabla u}^q\, dw}\right)^\frac{1}{q}\\
         \geqslant& \varlimsup_{n \to \infty}\left(\int_E{a_n(x)\abs{\nabla u_n - \nabla u}^q\, dw}\right)^\frac{1}{q}.
     \end{aligned}
\end{equation*}
We conclude the proof of the lemma.
\end{proof}
\begin{proof}[Proof of Theorem \ref{theoremmain}]
  Let $\left\{v_n\right\} \subset D^{1,\mathcal{H}}(\mathbb{Z}^N)$ be a minimizing sequence of \eqref{eqconstrained} and ${x_n} \subset \mathbb{Z}^N$ be such that $\abs{v_n(x_n)} = \norm{v_n}_\infty$.

We first excludes the vanishing case. Let $u_n(x) := v_n(x+y_n)$ with ${y_n}$ to ensure that $\left\{x_n-y_n\right\} \subset \Omega$, where $\Omega$ is an arbitrary bounded domain in $\mathbb{Z}^N$.
By Lemma \ref{lemlb} and \ref{lemkey}, we have
\[
    \varliminf_{n \to \infty}{\abs{u_n(x_n-y_n)}}=
 \varliminf_{n \to \infty}{\norm{v_n}_\infty} > 0,
\]
whence it follows that $u_n \to u \neq 0$ pointwise in $\mathbb{Z}^N$ since $\Omega$ is bounded.

 We claim that $\norm{u}_r=1$, i.e., $u \in M_{\frac{1}{r}}$. Suppose that this is not true, i.e.,  $0 < \norm{u}_r < 1$.
Hence, by Br\'ezis-Lieb lemma, we obtain $0 < \lim\limits_{n \to \infty}\norm{u_n-u}_r < 1$ and $\norm{u}_r^r +  \lim\limits_{n \to \infty}\norm{u_n-u}_r^r = 1$.
Since $r \geqslant q > p$, we get $\norm{u}_r^p > \norm{u}_r^q$, $\lim\limits_{n \to \infty}\norm{u_n-u}_r^p > \lim\limits_{n \to \infty}\norm{u_n-u}_r^q$ and $\lim\limits_{n \to \infty}\norm{u_n-u}_r^q + \norm{u}_r^q \geqslant 1$.
Moreover, by \eqref{eqtrans}, \eqref{eqIuv}, \eqref{eqdct} and Br\'ezis-Lieb lemma,  passing to a subsequence if necessary, we have
\begin{equation}
  \label{eqbi}
\begin{aligned}
  S
  = & \lim_{n \to \infty}{\int_E{\left(\frac{1}{p}\abs{\nabla u_n}^p + \frac{a_n(x)}{q}\abs{\nabla u_n}^q\right)}\, dw}\\
  = & \lim_{n \to \infty}{\int_E{\left(\frac{1}{p}\abs{\nabla (u_n-u)}^p + \frac{1}{p}\abs{\nabla u}^p + \frac{1}{q}\abs{a_n(x)^\frac{1}{q}\nabla u_n - \tilde{a}(x)^\frac{1}{q}\nabla u)}^q + \frac{\tilde{a}(x)}{q}\abs{\nabla u}^q\right)}\, dw}\\
  \geqslant & \lim_{n \to \infty}\left({\frac{\int_E{\abs{\nabla (u_n-u)}^p}\, dw}{p \norm{u_n-u}_r^p}} \norm{u_n-u}_r^p
  + \frac{\int_E{a_n(x)\abs{\nabla (u_n-u)}^q}\, dw}{q \norm{u_n-u}_r^q} \norm{u_n-u}_r^q \right)\\
  & + {\frac{\int_E{\abs{\nabla u}^p}\, dw}{p \norm{u}_r^p}} \norm{u}_r^p
  + \lim_{n \to \infty}\frac{\int_E{a_n(x)\abs{\nabla u}^q}\, dw}{q \norm{u}_r^q} \norm{u}_r^q\\
  > & S \left( \lim_{n \to \infty}\norm{u_n-u}_r^q + \norm{u}_r^q \right)\\
  \geqslant & S
\end{aligned}
\end{equation}
which is an contradiction and we prove that $\norm{u}_r=1$. Moreover, \eqref{eqbi} yields
\begin{equation*}
  S
  \geqslant \frac{1}{p}\lim_{n \to \infty}{\int_E{ \abs{\nabla (u_n-u)}^p\, dw}} + \lim_{n \to \infty}{\int_E{\left(\frac{1}{p}\abs{\nabla u}^p+ \frac{a_n(x)}{q}\abs{\nabla u}^q\right)}\, dw}
  \geqslant \frac{1}{p}\lim_{n \to \infty}{\int_E{ \abs{\nabla (u_n-u)}^p\, dw}} + S.
\end{equation*}
Hence, we conclude that $u_n \to u$ in $D^{1,\mathcal{H}}(\mathbb{Z}^N)$ by Remark \ref{remarkequi}.  Therefore, by using a variant of Lebesgue dominated convergence theorem, Proposition \ref{proembed} (ii)
implies
 \begin{equation}
 \label{eqanatilde}
     \lim_{n \to \infty}\int_E{a_n(x)\abs{\nabla u_n}^{q}}\, dw
  =  \int_E{\tilde{a}(x)\abs{\nabla u}^{q}}\, dw.
 \end{equation}
Hence $u$ satisfies \eqref{eqT1}.

    Furthermore, by  Ekeland's variational principle (see \cite{ekeland}), we obtain $I(v_n) \to S$ and
\begin{equation}
    \label{eqdiff}
    \lim_{n \to \infty}\left( I'(v_n) - \lambda_n J'(v_n) \right) = 0,
\end{equation}
where $\{\lambda_n\} \subset \mathbb{R}$ is a Lagrange multiplier sequence. Thus, by \eqref{eqanatilde} and the boundedness of  $\left\{ v_n \right\}$ in $D^{1,\mathcal{H}}(\mathbb{Z}^N)$,  we obtain
\begin{equation*}
    \begin{aligned}
    \lim_{n \to \infty}\lambda_n
    &= \lim_{n \to \infty}\lambda_n \int_{\mathbb{Z}^N}\abs{v_n}^{r}\, d\mu\\
    &= \lim_{n \to \infty}\left( \int_E{\abs{\nabla u_n}^{p}}\, dw + \int_E{a_n(x)\abs{\nabla u_n}^{q}}\, dw \right)\\
    &= \int_E{\abs{\nabla u}^{p}}\, dw + \int_E{\tilde{a}(x)\abs{\nabla u}^{q}}\, dw.
    \end{aligned}
\end{equation*}
Let $\lambda := \lim\limits_{n \to \infty}\lambda_n$. It follows from $r > p^*$, $\norm{u}_r=1$ and the Sobolev inequality \eqref{eqsobolev} that $\lambda > 0$. For any $\varphi \in C_0(\mathbb{Z}^N)$, we have $\left\{ \varphi_n(x) := \varphi(x - y_n)\right\}$ is bounded in $D^{1,\mathcal{H}}(\mathbb{Z}^N)$ by Remark \ref{remarkequi}. Let
\[
       \langle \tilde{I}'(u),\varphi \rangle := \int_E{\abs{\nabla u}^{p-2}{\nabla u}{\nabla \varphi}}\, dw
  + \int_E{\tilde{a}(x)\abs{\nabla u}^{q-2}{\nabla u}{\nabla \varphi}}\, dw.
\]
Then it follows from \eqref{eqdiff} that
\begin{equation*}
    \begin{aligned}
    \langle \tilde{I}'(u),\varphi \rangle
  &=     \lim_{n \to \infty}\left( \int_E{\abs{\nabla u_n}^{p-2}{\nabla u_n}{\nabla \varphi}}\, dw
  + \int_E{a_n(x)\abs{\nabla u_n}^{q-2}{\nabla u_n}{\nabla \varphi}}\, dw \right)\\
  &=  \lim_{n \to \infty}I'(v_n)\varphi_n\\
  &= \lim_{n \to \infty}\lambda_n J'(v_n)\varphi_n\\
  &= \lambda \lim_{n \to \infty} \int_{\mathbb{Z}^N}\abs{u_n}^{r-2}u_n\varphi\, d\mu\\
  &= \lambda \int_{\mathbb{Z}^N}\abs{u}^{r-2}u\varphi\, d\mu,
    \end{aligned}
\end{equation*}
 that is, $u$ is a nontrivial weak solution of \eqref{eqtildea}. It is clear that $u$ is a pointwise solution of \eqref{eqtildea}. Similar to the arguments in Theorem \ref{theoremcompact}, one can show that if $v_n$ are non-negative, then $u$ is positive.
\end{proof}

Now we provide the proof of Corollary \ref{corocases}.
\begin{proof}[Proof of Corollary \ref{corocases}]
    The proof is split in two steps.
   \begin{enumerate}[label=(\roman*)]
       \item $a$ is $T$-periodic.

        Let $\{v_n\} \subset D^{1,\mathcal{H}}(\mathbb{Z}^N)$ be a minimizing sequence of \eqref{eqconstrained} and ${x_n} \subset \mathbb{Z}^N$ be such that $\abs{v_n(x_n)} = \norm{v_n}_\infty$. Since $\{\abs{v_n}\}$ is also a minimizing sequence, we may assume that $v_n$ are non-negative.

        Let $\{k_n:=(k_n^1,...,k_n^N)\} \subset \mathbb{Z}^N$ to ensure that $\{x_n - Tk_n\} \subset \Omega$ where $T \in \mathbb{Z}$, $\Omega = [0,T)^N\cap\mathbb{Z}^N$. Let $\left\{ u_n(x) := v_n(x+Tk_n)\right\}$. Then, by Theorem \ref{theoremmain}, we have $u_n \to u$ in $D^{1,\mathcal{H}}(\mathbb{Z}^N)$, where $u$ satisfies \eqref{eqT1} and solves \eqref{eqtildea}. Thus $u$ is a positive minimizer for \eqref{eqconstrained} and pointwise solution of \eqref{eq2}, since $\tilde{a}(x) = \lim\limits_{n \to \infty}a(x+Tk_n) = a(x)$.

        \item $a$ is a bounded potential.

        We consider the case that $a \not\equiv a_\infty := \lim\limits_{\abs{x} \to \infty} a(x)$, since the case that $a$ is a constant is contained in the proof above. Then there exists $x_0 \in \mathbb{Z}^N$ such that $a(x_0) < a_\infty$.
        Let
        \begin{equation}
        \label{eqinfty}
            S_\infty := \inf_{u \in M_{\frac{1}{r}}} \int_E{\left(\frac{1}{p}\abs{\nabla u}^p + \frac{a_\infty}{q}\abs{\nabla u}^q\right)}\, dw.
        \end{equation}
        We claim that $S < S_\infty$. In fact, it follows from the first part of this proof that problem \eqref{eqinfty} has a positive solution $u_\infty$. Then there exists $x_1y_1 \in E$ such that $(\nabla u_\infty)(x_1,y_1) \neq 0$. Hence, for $v(x) := u_\infty(x+x_1-x_0)$, we have
        \[
            S_\infty = \int_E{\left(\frac{1}{p}\abs{\nabla v}^p + \frac{a_\infty}{q}\abs{\nabla v}^q\right)}\, dw >  \int_E{\left(\frac{1}{p}\abs{\nabla v}^p + \frac{a(x)}{q}\abs{\nabla v}^q\right)}\, dw \geq S.
        \]
Now, let $\{v_n\} \subset D^{1,\mathcal{H}}(\mathbb{Z}^N)$ be a minimizing sequence of \eqref{eqconstrained} and ${x_n} \subset \mathbb{Z}^N$ be such that $\abs{v_n(x_n)} = \norm{v_n}_\infty$. Since $\{\abs{v_n}\}$ is also a minimizing sequence, we may assume that $v_n$ are non-negative. Let $\left\{ u_n(x) := v_n(x+x_n)\right\}$. Then, by Theorem \ref{theoremmain}, we have $u_n \to \tilde{u}$ in $D^{1,\mathcal{H}}(\mathbb{Z}^N)$ and
        \begin{equation}
            \label{eqbounded}
             S =  \int_E{\left(\frac{1}{p}\abs{\nabla \tilde{u}}^p + \frac{\tilde{a}(x)}{q}\abs{\nabla \tilde{u}}^q\right)}\, dw.
        \end{equation}

        Suppose $\abs{x_n} \to \infty$ as $n \to \infty$. Then $\tilde{a}(x) = \lim\limits_{n \to \infty}a_n(x+x_n) = a_\infty$. By \eqref{eqbounded}, we get
        \[
            S =  \int_E{\left(\frac{1}{p}\abs{\nabla \tilde{u}}^p + \frac{a_\infty}{q}\abs{\nabla \tilde{u}}^q\right)}\, dw \geqslant \int_E{\left(\frac{1}{p}\abs{\nabla \tilde{u}}^p + \frac{a(x)}{q}\abs{\nabla \tilde{u}}^q\right)}\, dw
            \geqslant S_\infty,
        \]
         which is a contradiction to $S < S_\infty$. Thus $\{x_n\}$ is bounded.

         Passing to a subsequence if necessary, let $x^* := \lim\limits_{n \to \infty} x_n$,  $\tilde{a}(x) = \lim\limits_{n \to \infty}a_n(x+x_n) = a(x+x^*)$. Let $u := \tilde{u}(x-x^*)$, we have
         \[
            \int_E{\left(\frac{1}{p}\abs{\nabla u}^p + \frac{a(x)}{q}\abs{\nabla u}^q\right)}\, dw
            = \int_E{\left(\frac{1}{p}\abs{\nabla \tilde{u}}^p + \frac{a(x+x^*)}{q}\abs{\nabla \tilde{u}}^q\right)}\, dw
            = S.
         \]
         Thus $u$ is a non-negative minimizer for \eqref{eqconstrained}.
         Similar to the arguments in Theorem \ref{theoremcompact}, one can show that $u$ is a positive pointwise solution of \eqref{eq2}.
   \end{enumerate}
\end{proof}

\section{Proof of Theorem \ref{theoremeg}}\label{sectioneg}
In the final section, by adapting the arguments of \cite[Theorem 1.2]{Fan}, we shall study the asymptotic behaviour of Lagrangian multiplier $\lambda$ with respect to $t$. It follows from \eqref{eq2} that
\[
    \lambda\left( u_t \right) = \frac{ \int_E{\abs{\nabla u_t}^{p}\, dw}
  + \int_E{a(x)\abs{\nabla u_t}^{q}\, dw}}{\int_{\mathbb{Z}^N}\abs{u_t}^{r}\, d\mu}.
\]
Since $I(u_t) = S_t$ and $\int_{\mathbb{Z}^N}\abs{u_t}^{r} = rt$, we have
\[
      \frac{pS_{t}}{rt} \leqslant \lambda\left( u_t \right) \leqslant \frac{qS_{t}}{rt}.
\]
Hence, for the proof of Theorem \ref{theoremeg}, it suffices to show
\[
    \frac{S_t}{t} \to 0 \text{ as } t \to \infty, \text{ and } \frac{S_t}{t} \to \infty \text{ as } t \to 0.
\]
\begin{proof}[Proof of theorem \ref{theoremeg}]
The proof is split in two steps.
\begin{enumerate}[label=(\roman*)]
\item Let $r > q$.

First, fix $y \in \mathbb{Z}^N$. Thus, for $t > 1$, we have
\[
\frac{S_t}{t} \leqslant \frac{I\left((rt)^\frac{1}{r}\delta_{y}\right)}{t} \leqslant (rt)^{\frac{q}{r}-1} I\left(\delta_y\right).
\]
Hence $\frac{S_t}{t} \to 0$ as $t \to \infty$ because $r > q$.

Now, let $t \in(0,1)$. Take $u_t \in A_t$. Then, by \red{the} Sobolev inequality \eqref{eqsobolev}, we obtain
\[
    \int_E{\abs{\nabla u_t}^{p}\, dw} \geqslant C_r^pt^\frac{p}{r}.
\]
Thus
\[
    \frac{S_t}{t} = \frac{I(u_t)}{t} \geqslant \frac{{\frac{1}{p}}  \int_E{\abs{\nabla u_t}^{p}\, dw}}{t} \geqslant \frac{1}{p}C_r^pt^{\frac{p}{r}-1},
\]
hence $\frac{S_t}{t} \to \infty$ as $t \to 0$ because $r > p^* > p$.
\item  Let $r = q$ and $a$ satisfies the hypotheses of Corollary \ref{corocases}.

Since $r > p$ still holds, we only need to show
\[
    \frac{S_t}{t} \to \infty \text{ as } t \to 0.
\]
Let
\[
    w_n(x):=
     \begin{cases}
       n^{-\frac{N}{r}},  & x \in \left[-\frac{n}{2}, \frac{n}{2}\right)^N \cap \mathbb{Z}^N,\\
       0, & else.
      \end{cases}
\]
Then $\int_{\mathbb{Z}^N}\abs{w_n}^{r} = 1$ and
\[
    I(w_n) \leqslant \frac{2}{p}Nn^{N-1}n^{-\frac{Np}{r}} + \frac{2}{q}\norm{a}_\infty Nn^{N-1}n^{-\frac{Nq}{r}} \leqslant C(n^{N(1-\frac{p}{r})-1} + n^{-1}),
\]
Thus, for $t > 1$ and $t^{\frac{1}{N}}-1 \leqslant n < t^{\frac{1}{N}}$,
\[
\frac{S_t}{t} \leqslant \frac{I\left((rt)^\frac{1}{r}w_n\right)}{t} \leqslant  \frac{C_1(t^{\frac{p}{r}}n^{N(1-\frac{p}{r})-1} + t^{\frac{q}{r}}n^{-1})}{t} < C_2(t^{- \frac{1}{N}} + (t^{\frac{1}{N}}-1)^{-1}).
\]
Hence $\frac{S_t}{t} \to 0$ as $t \to \infty$.
\end{enumerate}
\end{proof}
\begin{remark} \label{remarkegex}
Under the hypotheses of Theorem \ref{theoremcompact} with $r \leq q$, our second conclusion, i.e., $\lambda(u_t) \to \infty \text{ as } t \to 0^{+}$ holds since we only use the fact that $r > p$ in the proof. The same arguments above also works for $\lambda$ defined in Theorem \ref{theoremmain}.
\end{remark}
\begin{remark}
Theorem \ref{theoremeg} and Remark \ref{remarkegex} can be extended to more general graphs. Let  $G = (V, E)$ be a uniformly locally finite graph satisfying $\inf\limits_{x \in V}\mu(x) > 0$, $\inf\limits_{xy \in E}w_{xy} > 0$, and $d$-isoperimetric inequality for some $d \geqslant 2$. If $a$ satisfies $\lim_{d(x,x_0) \to \infty}a(x) = + \infty$ for some $x_0 \in V$, then $\lambda(u_t) \to 0 \text{ as } t \to +\infty$ holds for all $r>q$,  and $\lambda(u_t) \to \infty \text{ as } t \to 0^{+}$ holds for all $r>p$. If we further assume that $G$ is a quasi-transitive graph with $\mu \equiv 1$ and $w \equiv 1$, then Theorem \ref{theoremeg} and Remark \ref{remarkegex} hold for all $r > q$.
\end{remark}

\subsection*{Acknowledgements}
C. Ji was partially supported by National Natural Science Foundation of China (No. 12171152).



  
\end{document}